\documentclass[12pt]{article}
\usepackage{amsmath}
\usepackage{latexsym}
\usepackage{amssymb}
\newtheorem{thm}{Theorem}[section]
\newtheorem{la}[thm]{Lemma}
\newtheorem{Defn}[thm]{Definition}
\newtheorem{Remark}[thm]{Remark}
\newtheorem{prop}[thm]{Proposition}

\newtheorem{Number}[thm]{\!\!}

\newenvironment{rem}{\begin{Remark}\rm}{\end{Remark}}
\newenvironment{numba}{\begin{Number}\rm}{\end{Number}}
\newenvironment{proof}{{\noindent\bf Proof.}}%
                  {\nopagebreak\hspace*{\fill}$\Box$\medskip\medskip\par}   
\newcommand{\Punkt}{\nopagebreak\hspace*{\fill}$\Box$}
\newcommand{\wb}{\overline}
\newcommand{\ve}{\varepsilon}
\newcommand{\at}{\symbol{'100}}

\newcommand{\wt}{\widetilde}

\newcommand{\mto}{\mapsto}

\newcommand{\isom}{\cong}
\newcommand{\N}{{\mathbb N}}
\newcommand{\R}{{\mathbb R}}
\newcommand{\K}{{\mathbb K}}

\newcommand{\Z}{{\mathbb Z}}
\newcommand{\C}{{\mathbb C}}
\newcommand{\cO}{{\mathcal O}}
\newcommand{\cB}{{\mathcal B}}
\newcommand{\sub}{\subseteq}

\DeclareMathOperator{\im}{im}

\DeclareMathOperator{\id}{id}

\newcommand{\obs}{{\scriptstyle{\rm s}}}

\newcommand{\obu}{{\scriptstyle{\rm u}}}

\DeclareMathOperator{\Lip}{Lip}

\DeclareMathOperator{\spread}{spread}
\begin{document}
\begin{center}
{\Large\bf Grobman-Hartman Theorems\\[.8mm]
for Diffeomorphisms of Banach Spaces\\[2.6mm]
over Valued Fields}\\[6mm]
\renewcommand{\thefootnote}{\fnsymbol{footnote}} 
{\bf Helge Gl\"{o}ckner\footnote{The research was carried out at PUC,
Santiago de Chile, supported by
the German\linebreak
Research Foundation 
(DFG), grant GL 357/8-1
and the
`Research Network on Low\linebreak
Dimensional Systems,' PBCT/CONICYT, Chile.
Further work on the manuscript was carried out at the University of Newcastle (N.S.W.), again supported by GL 357/8-1.}}\vspace{4mm} 
\end{center} 
\renewcommand{\thefootnote}{\arabic{footnote}} 
\setcounter{footnote}{0} 
\begin{abstract}
\hspace*{-6mm}Consider a local diffeomorphism $f$
of an ultrametric Banach space over an ultrametric field,
around a hyperbolic fixed point $x$.
We show that, locally, the system is
topologically conjugate to the linearized system.
An analogous result is obtained for local diffeomorphisms
of real $p$-Banach spaces (like $\ell^p$), for $p\in \;]0,1]$.
More generally, we obtain a local linearization
if $f$ is merely a local homeomorphism
which is strictly differentiable at a hyperbolic fixed point~$x$.
Also a new global version of the
Grobman-Hartman
\mbox{theorem} is provided. It applies to Lipschitz perturbations
of hyperbolic automorphisms of Banach spaces over valued fields.
The local conjugacies $H$ constructed are not only homeo\-morphisms,
but both $H$ and $H^{-1}$ are H\"{o}lder.
We also study the dependence of $H$ and $H^{-1}$ on $f$
(keeping $x$ and $f'(x)$ fixed).
\end{abstract}
{\footnotesize {\em Classification}:
37C15 (Primary)
26E30,
46A16,
46S10
(Secondary)\\[2.5mm]
}\\[3.4mm]
\section{\hspace*{-3mm}Introduction and statement of main results}\label{secintro}
The linearization problem
for formal or analytic diffeomorphisms
of a complete ultrametric field $\K$ (or $\K^n$),
via formal or analytic conjugacies,
has attracted interest in non-archimedean
analysis (see \cite{HaY}, \cite{Lin} and \cite{Vie}).
Since an analytic linearization is
not always possible,
it is natural to ask whether at least
a (local) topological conjugacy from the given system to its linearized
version is available.
In the current article, we answer this question in the affirmative
(under natural hyperbolicity hypotheses).\\[2.5mm]
More generally, for some of our results
we can work with
a valued field $(\K,|.|)$
whose absolute value
$|.|$ is assumed to define a non-discrete topology on $\K$
(such a field is called ultrametric if
$|x+y|\leq\max\{|x|,|y|\}$ for all $x,y\in \K$).
If $E$ is a Banach space over $\K$,
we shall say that an automorphism $A\colon E\to E$
of topological vector spaces is
\emph{hyperbolic} if
if there exist $A$-invariant
vector subspaces $E_{ \obs }$ and $E_{ \obu }$ of $E$ such that
$E=E_{ \obs } \oplus E_{ \obu }$,
and a norm $\|.\|$ on $E$ defining its topology, such that
\begin{equation}\label{hypb1}
\|x+y\|=\max\{\|x\|,\|y\|\}\quad
\mbox{for all $x\in E_{\obs}$ and $y\in E_{\obu}$}
\end{equation}
and
\begin{equation}\label{hypb2}
\|A|_{E_{\obs}}\|< 1\quad\mbox{ and }\quad \|A^{-1}|_{E_{\obu}}\|<1
\end{equation}
holds for the operator norms with respect to $\|.\|$
\,(then call $\|.\|$ \emph{adapted to $A$}).\\[2.5mm]
Our two main theorems
are versions of the local and global Grobman-Hartman theorem
for $C^1$-diffeomorphisms of $\R^n$
(see \cite{Gr1}, \cite{Gr2}, \cite{Ha1}, \cite{Ha2}, \cite{Irw},
\cite{PdM} and \cite{Rob} for these classical results
and their analogues for flows). Our presentation is particularly
indebted to~\cite{Rob}. We first discuss global conjugacies: \\[2.5mm]
{\bf Theorem A} (Global Grobman-Hartman Theorem)
\emph{Let $E$ be a Banach space over a valued field $(\K,|.|)$
and $A \colon E\to E$ be an automorphism
of topological vector spaces which is hyperbolic.
Let $\|.\|\colon E\to [0,\infty[$ be a norm adapted to~$A$
and $g\colon E\to E$ be a bounded Lipschitz map
such that}
\[
\Lip(g)<\|A^{-1}\|^{-1},\quad
\|A^{-1}|_{E_\obu}\|(1+\Lip(g))<1,\quad\mbox{and}\quad
\|A|_{E_\obs}\|+\Lip(g)<1.
\]
\emph{Then there exists a unique
bounded continuous map
$v\colon E\to E$ such that}
\begin{equation}
(A+g)\circ (\id_E+v)\;=\; (\id_E+v)\circ A\,.
\end{equation}
\emph{The map $\id_E+v$ is a homeomorphism
from~$E$ onto~$E$, and both $v$ and
$w:=(\id_E+v)^{-1}-\id_E$ are H\"{o}lder.
Moreover, $w$ is the unique bounded
continuous map such that}
\begin{equation}
A \circ (\id_E+w)\;=\; (\id_E+w)\circ (A+g)\,.
\end{equation}
\emph{If $g(0)=0$, then also $v(0)=0$.}\\[2.5mm]
A H\"{o}lder exponent $\alpha$ for $v$ and $w$
can be described explicitly (Remark~\ref{remtherem} (a) and (b)).
See also \cite{BaV} for a recent discussion of the H\"{o}lder properties
of $v$ and $w$ in the real case (if $g(0)=0$).\\[2.5mm]
To obtain a local linearization, following
Hartman \cite{Ha2},
we shall only require
strict differentiability of $f$ at the fixed point.
Let $(E,\|.\|_E)$ and $(F,\|.\|_F)$
be Banach spaces over a valued field $(\K,|.|)$,
$U\sub E$ be open and $z\in U$.
We recall from Bourbaki \cite{Bo1}:
% STELLE
A map
\begin{equation}\label{stupideq}
f\colon U\to F
\end{equation}
is called \emph{strictly differentiable
at $x$} if there exists a (necessarily unique)
continuous linear map $f'(x)\colon E\to F$
such that
\begin{equation}
\frac{\|f(y)-f(z)-f'(x)(y-z)\|_F}{\|y-z\|_E}\; \to \; 0
\end{equation}
if $(y,z)\in U\times U\setminus \{(u,u)\colon u\in U\}$ tends to $(x,x)$.
If we write $f(y)=f(x)+f'(x)(y-x)+R(y)$,  then $f$ is strictly differentiable at $x$ with derivative $f'(x)$
if and only if $R$ is Lipschitz on the ball $B^E_r(x)$ for
small $r>0$, and
\begin{equation}\label{stricvialip}
\lim_{r\to 0} \, \Lip(R|_{B^E_r(x)})\; =\; 0
\end{equation}
(using standard notation as in \ref{defbaselip}).\footnote{Precisely this
requirement on the non-linearity is also imposed in \cite{Ha2}.}\\[2.5mm]
If $E=F$ and $f$ is strictly differentiable at $x\in U$,
we call $x$ a \emph{hyperbolic fixed point of $f$}
if $f(x)=x$ and $f'(x)\colon E\to E$ is a hyperbolic automorphism.\\[2.5mm]
{\bf Theorem B} (Local Grobman-Hartman Theorem)
\emph{Let $(\K,|.|)$ be an ultrametric field
and $E$ be an ultrametric Banach space over $(\K,|.|)$.
Or let $\K=\R$,
$|.|$ be an absolute value on $\R$ which defines
the usual topology on $\R$,
and $E$ be a Banach space over $(\R,|.|)$.
Let $P,Q\sub E$ be open
and $x \in P\cap Q$. Let
$f\colon P\to Q$ be a homeomorphism
which is strictly differentiable at $x$,
with differential $A:=f'(x)$,
and for which $x$ is a hyperbolic fixed point.
Then there exists an open $0$-neighbourhood $U\sub E$
and a bi-H\"{o}lder homeomorphism $H\colon U\to V$ onto an open
subset $V\sub P$, such that $H(0)=x$ and}
\begin{equation}\label{goaleq}
f(H(y))=H(A(y))\qquad \mbox{for all $\; y\in U\cap A^{-1}(U)$.}\vspace{1mm}
\end{equation}
Recall that the absolute values $|.|$ on $\R$
defining its usual topology are precisely
the $p$-th powers of the usual absolute value $|.|_\R$,
i.e., $|.|=(|.|_\R)^p$, with $p\in\;]0,1]$.
The Banach spaces
$E$ over $(\R,|.|_\R^p)$
are also known as real \emph{$p$-Banach spaces}
in the functional-analytic literature
(see \cite{Jar}).\\[2.5mm]
To deduce Theorem B from Theorem A,
we shall cut off the nonlinearity.
Since suitable cut-offs only come to mind
in the real and ultrametric cases,
we have to restrict attention to these situations.\\[2.5mm]
We also discuss the dependence of the conjugacies $\id_E+v$ (and $\id_E+w$)
on~$f$.
In the global case,
we obtain
Lipschitz resp.\ H\"{o}lder continuous dependence
of $v$ (resp., $w$) as elements in the space $BC(E,E)$
of bounded continuous functions,
with respect to the supremum norm
(Theorem~\ref{globparthm}).
Similar results are obtained for the local conjugacies from
Theorem~B;
in this case, we also obtain continuous dependence
of $H$ and $H^{-1}$ when considered as elements of appropriate
H\"{o}lder spaces, if $\K$ is locally compact
and $\dim(E)<\infty$ (see Proposition~\ref{pardepno2}).
For earlier results concerning parameter dependence
in the real case, the reader is referred to
\cite[Theorem 26]{IRCOMP}.\\[2.5mm]
To put the
requirement of strict differentiability into context, we recall:
If $\K=\R$, equipped with its usual absolute value,
then $f$ as in (\ref{stupideq}) is strictly differentiable
at each $x\in U$ if and only if $f$ is continuously
Fr\'{e}chet differentiable (\cite[2.3.3]{Bo1}, \cite[Theorem 3.8.1]{Car}).
If $(\K,|.|)$ is arbitrary and $f$ is $C^2$
in the sense of \cite{Ber}, then $f$ is strictly differentiable at
each $x$ \cite[Proposition 3.4]{IMP}.
If $(\K,|.|)$
is a complete ultrametric field
and $E$ of finite dimension,
then $f$ is strictly differentiable at
each $x$ if and only if $f$ is $C^1$ in the sense
of~\cite{Ber}
(see \cite[Appendix C]{FIO}),
hence if and only if it is $C^1$ in the usual
sense of finite-dimensional non-archimedean analysis
(as in \cite{Sch}, \cite{dSm});
see~\cite{CMP}.\\[2.5mm]
In the classical real case,
it is known that conjugacies cannot be chosen
locally Lipschitz in general (see~\cite{Bel}, cited from \cite{Str}).
In particular, they need not be $C^1$ (although the $C^1$-property
-- and higher differentiablity properties --\linebreak
can be guaranteed under
suitable non-resonance conditions~\cite{Ste}).
The investigation of the possible continuity and differentiability
properties of local conjugacies (e.g., differentiability at the fixed point)
remains an active
area of research (see \cite{BaV}, \cite{GHR}, \cite{RS1}, \cite{RS2}, \cite{Str}
for some recent work).
The current article provides a foundation
for a later study of such refined questions also in the
non-archimedean case.\\[2.5mm]
The above concept of hyperbolicity is useful
also for other ends.
For example, as in the real case,
a stable manifold
can be constructed around each hyperbolic fixed point
(if $f$ is
analytic and the adapted norm is ultrametric)~\cite{STA}.
\section{Preliminaries and notation}
We fix some notation
and compile facts and preparatory results for later
use.
\begin{numba}
Given a metric space $(X,d)$,
$r>0$ and $x\in X$,
we define $B^X_r(x):=\{y\in X\colon d(x,y)<r\}$
and $\wb{B}^X_r(x):=\{y\in X\colon d(x,y)\leq r\}$.
As usual, a normed space $(E,\|.\|)$ over a valued field $(\K,|.|)$
is called a Banach space if it is complete.
If, moreover,  $(\K,|.|)$ is an ultrametric field and also
$\|.\|$ satisfies the ultrametric inequality $\|x+y\|\leq \max\{\|x\|,\|y\|\}$,
then $(E,\|.\|)$ is called an \emph{ultrametric} Banach space (see
\cite{Roo} for further information).
The ultrametric inequality implies that
\begin{equation}\label{stronger}
\|x+y\|=\|y\|\quad\mbox{for all $\,x,y\in E$ such that $\|x\|<\|y\|$.}
\end{equation}
If $(E,\|.\|_E)$ and $(F,\|.\|_F)$ are normed spaces over a valued field $(\K,|.|)$
and $A\colon E\to F$ a continuous linear map,
then its operator norm is defined as $\|A\|:=\sup\{\|Ax\|_F/\|x\|_E\colon 0\not=x\in E\}\in [0,\infty[$.
\end{numba}
\begin{numba}
If $f\colon X\to E$ is a bounded map to
a normed space $(E, \|.\|)$ over a valued field $(\K,|.|)$,
we write $\|f\|_\infty:=\sup\{\|f(x)\|\colon x\in X\}$
for its supremum norm.
Given a topological space $X$, we write $BC(X,E)$
for the set of bounded, continuous functions from $X$
to $E$. This is a normed space with respect to the supremum norm,
and a Banach space (respectively, an ultrametric Banach space)
if so is $E$.
\end{numba}
\begin{numba}\label{defbaselip}
As usual, we call a map $f\colon X\to Y$ between
metric spaces $(X,d_X)$ and $(Y,d_Y)$
(globally) \emph{H\"{o}lder} of exponent $\alpha\in \;]0,\infty[$
if there exists $L\in [0,\infty[$
such that $d_Y(f(x),f(y))\leq L\,d_X(x,y)^\alpha$ for all
$x,y\in X$. We let $\Lip_\alpha(f)$ be the minimum choice
of~$L$. If $f$ is bijective and both $f$ and $f^{-1}$
are H\"{o}lder (of exponent $\alpha$),
we call $f$ a bi-H\"{o}lder homeo\-morphism (of exponent $\alpha$).
H\"{o}lder maps of exponent $1$ are called Lipschitz
and we abbreviate $\Lip(f):=\Lip_1(f)$.
Thus $\Lip(A)=\|A\|$ for continuous linear maps.
We write $L_\alpha(X,Y)$ for the set of all
H\"{o}lder maps $f\colon X\to Y$ of exponent~$\alpha$.
If $(E,\|.\|_E)$
is a Banach space over a valued field $(\K,|.|)$,
then also
\[
BL_\alpha(X,E):=L_\alpha(X,E)\cap BC(X,E)
\]
is a Banach space, with respect to the norm
$\|f\|_\alpha:=\max\{\|f\|_\infty,\Lip_\alpha(f)\}$.
If, moreover, $X$ is compact, then $BL_\alpha(X,E)=L_\alpha(X,E)$
and thus $\|.\|_\alpha$ makes $L_\alpha(X,E)$ a Banach space.
\end{numba}
\begin{la}\label{basecom}
Let $(X,d_X)$, $(Y,d_Y)$ and $(Z,d_Z)$ be metric spaces,
$f\colon X\to Y$ be H\"{o}lder of exponent $\alpha$,
and $g\colon Y\to Z$ be H\"{o}lder of exponent $\beta$.
Then $g\circ f\colon X\to Z$ is H\"{o}lder of exponent $\alpha\beta$,
and
\[
\Lip_{\alpha\beta}(g\circ f)\; \leq\; \Lip_\beta(g)\, (\Lip_\alpha(f))^\beta\,.
\]
\end{la}
\begin{proof}
For $x,y\in X$, we have $d_Z(g(f(x)),g(f(y)))\leq \Lip_\beta(g)d_Y(f(x),f(y))^\beta$
$\leq \Lip_\beta(g)\Lip_\alpha(f)^\beta d_X(x,y)^{\alpha \beta}$.
\end{proof}
\begin{la}\label{desc}
Let $(X,d_X)$,
$(Y,d_Y)$ be metric spaces,
$\alpha\geq \beta> 0$
and \mbox{$f\colon X\to Y$} be a
H\"{o}lder map
of exponent $\alpha$, which is bounded
in the sense that
\[
\spread(f)\, :=\, \sup\{d_Y(f(x),f(y))\colon x,y\in X\}<\infty \, .
\]
Then $f$ is also
H\"{o}lder
of exponent $\beta$,
and
\begin{equation}\label{lowerhold}
\Lip_\beta(f)\leq \max\{\Lip_\alpha(f),\spread(f)\}\,.
\end{equation}
\end{la}
\begin{proof}
Let $x,y\in X$.
If $d_X(x,y)\leq 1$,
then
\begin{equation}\label{to1}
d_Y(f(x),f(y))\leq \Lip_\alpha(f)\,  d_X(x,y)^\alpha\leq
\Lip_\alpha(f)\,  d_X(x,y)^\beta\,.
\end{equation}
If $d_X(x,y)\geq 1$,
then
\begin{equation}\label{to2}
d_Y(f(x),f(y))\leq \spread(f) \leq
\spread(f) \, d_X(x,y)^\beta\,.
\end{equation}
Now (\ref{lowerhold}) follows from (\ref{to1})
and (\ref{to2}).
\end{proof}
\begin{la}\label{compohoeld}
Let $(E,\|.\|_E)$ and $(F,\|.\|_F)$
be normed spaces over a valued field $(\K,|.|)$,
$h\colon E\to F$ be a bounded Lipschitz map
and $v\colon E\to E$ be a map which is H\"{o}lder
of some exponent $\alpha\in \;]0,1]$.
Then also the map $h\circ (\id_E+v)\colon E\to F$
is H\"{o}lder of exponent~$\alpha$,
and
\[
\Lip_\alpha(h\circ (\id_E+ v))\;\leq\; \max\{\Lip(h)(1+\Lip_\alpha(v)),\spread(h)\}\,.
\]
In particular, $\Lip_\alpha(h\circ (\id_E+ v))
\leq\max\{\Lip(h)(1+\Lip_\alpha(v)), 2\|h\|_\infty \}$.
\end{la}
\begin{proof}
Let $x,y\in E$. If $\|y-x\|_E \leq 1$,
then $\|y-x\|_E\leq \|y-x\|^\alpha_E$ and hence
\begin{eqnarray*}
\|h(y+v(y))-h(x+v(x))\|_F & \leq &
\Lip(h)\|y+v(y)-x-v(x)\|_E\\
&\leq &
\Lip(h)(\|y-x\|_E+\Lip_\alpha(v)\|y-x\|^\alpha_E)\\
&\leq &
\Lip(h)(1+\Lip_\alpha(v)) \|y-x\|^\alpha_E\, .
\end{eqnarray*}
If $\|y-x\|_E\geq 1$, we have
$\|h(y+v(y))-h(x+v(x))\|_F\leq\spread(h)\leq \spread(h)\|y-x\|_E^\alpha$.
The assertion follows from the preceding estimates.
\end{proof}
\begin{la}\label{pwplip}
Let $(E,\|.\|)$ be a normed space
over a valued field $(\K,|.|)$,
$(X,d)$ be a metric space
and $\xi\colon X\to\K$ and $f\colon X\to E$
be bounded, Lipschitz maps.
Then also the pointwise product $\xi f$
is bounded and Lipschitz, with
\[
\Lip(\xi f)\, \leq\, \Lip(\xi)\|f\|_\infty+\|\xi\|_\infty\Lip(f)\,.
\]
\end{la}
\begin{proof}
$\|\xi(y)f(y)\!-\!\xi(x)f(x)\| \leq
|\xi(y)\!-\!\xi(x)|\,\|f(y)\|\!+\!|\xi(x)|\,\|f(y)\!-\!f(x)\|$.
\end{proof}
\begin{la}\label{globinv}
Let $(E,\|.\|)$ be a Banach space over a valued field $(\K,|.|)$
$($such that $E\not=\{0\})$
and $A\colon E\to E$ be an automorphism of topological vector spaces.
Moreover, let $v\colon E\to E$ be a Lipschitz map
such that $\Lip(v)<\frac{1}{\|A^{-1}\|}$. Then the map
$f:=A+v\colon E\to E$ is a homeomorphism,
and $f^{-1} \colon E\to E$
is Lipschitz with
\begin{eqnarray}
\Lip(f^{-1}) & \leq & \frac{1}{\|A^{-1}\|^{-1}-\Lip(v)}
\quad \mbox{and}\label{lipforinv}\\
\Lip(f^{-1}-A^{-1}) &\leq & \frac{\|A^{-1}\|}{\|A^{-1}\|^{-1}-\Lip(v)}\Lip(v)\,.\label{lipdistinv}
\end{eqnarray}
If $v$ is bounded, then also $w:=f^{-1}-A^{-1}$ is bounded, and $\|w\|_\infty\leq \|A^{-1}\|\, \|v\|_\infty$.
\end{la}
\begin{proof}
Set $a:=\|A^{-1}\|^{-1}-\Lip(v)>0$.
By the Lipschitz Inverse Function\linebreak
Theorem
(see \cite[Theorem 5.3]{FIO}), the restriction
$f_r:=f|_{B_r^E(0)}$ is injective for
each $r>0$, whence $f$ is injective.
By the same theorem,
the inverse map $(f_r)^{-1}\colon f(B_r^E(0))\to E$ is Lipschitz
with $\Lip(f_r^{-1})\leq a^{-1}$.
Hence also $f^{-1}\colon f(E)\to E$
is Lipschitz, with $\Lip(f^{-1})\leq a^{-1}$,
and thus (\ref{lipforinv}) holds.
In particular, $f$ is a homeomorphism onto its image.
By the cited theorem, $f(B_r^E(0))\supseteq B_{ar}^E(f(0))$
for each~$r$. Hence $f(E)\supseteq \bigcup_{r>0}B_{ar}^E(f(0))=E$,
whence $f$ is surjective.
To complete the proof, write $w:=f^{-1}-A^{-1}$.
Then
$\id_E=(A+v)\circ (A+v)^{-1}=(A+v)\circ (A^{-1}+w)=
\id_E+A\circ w+v\circ (A^{-1}+w)
=
\id_E+A\circ w+v\circ f^{-1}$
and thus
\begin{equation}\label{enabiter}
w=-A^{-1}\circ v\circ f^{-1}.
\end{equation}
Hence $\Lip(w)\leq \Lip(A^{-1})\Lip(v)\Lip(f^{-1})=
\|A^{-1}\|\Lip(v)\Lip(f^{-1})$.
If we combine this estimate with (\ref{lipforinv}),
we obtain (\ref{lipdistinv}).
Finally, assuming that $v$ is bounded,
(\ref{enabiter}) shows
that also $w$ is bounded, with $\|w\|_\infty\leq \|A^{-1}\|\, \|v\|_\infty$.
\end{proof}
\section{Passage from one perturbation to another}\label{bforth}
In this section, we construct conjugacies from one perturbation of a given hyperbolic automorphism to another.
\begin{la}\label{unqbdd}
Let $E\not=\{0\}$ be a Banach space over a valued field $(\K,|.|)$,
$A\colon E\to E$ be a hyperbolic
automorphism, and $\|.\|$ be an adapted norm on $E$.
Let $g=(g_\obs,g_\obu)$, $h=(h_\obs,h_\obu)\colon E\to E=E_\obs\oplus E_\obu$
be bounded Lipschitz maps
such that
\begin{equation}\label{hyp1}
\Lip(h) <  \|A^{-1}\|^{-1}\qquad\mbox{and}\qquad
\end{equation}
\begin{equation}\label{hyp2}
\Lambda:=\max\big\{\|A^{-1}_2\| (1 +\Lip(g_\obu)), \|A_1\| +\Lip(g_\obs)\big\}\,<\, 1\, ,
\end{equation}
with $A_1:=A|_{E_\obs}\colon E_\obs\to E_\obs$ and $A_2:= A|_{E_\obu}$.
Then there exists a unique
bounded continuous map $v\colon E\to E$
such that
\begin{equation}\label{detbd}
(\id_E+v)\circ (A+h)\;=\; (A+g)\circ (\id_E+v)\,.
\end{equation}
It satisfies
\begin{equation}\label{unifestim}
\|v\|_\infty \,\leq\,
\frac{\max\{\|h_\obs\|_\infty + \|g_\obs\|_\infty,  \|A_2^{-1}\|(\|h_\obu\|_\infty+ \|g_\obu\|_\infty )\}}{1-\Lambda}\,.
\end{equation}
If $g(0)=h(0)=0$, then also $v(0)=0$.
\end{la}
\begin{proof}
As a consequence of (\ref{hyp1}),
$A+h\colon E\to E$ is a homeomorphism,
whose inverse $(A+h)^{-1}$ is Lipschitz with
$\Lip((A+h)^{-1})\leq (\|A^{-1}\|^{-1}-\Lip(h))^{-1}$
(see Lemma~\ref{globinv}).
For a bounded continuous function $v\colon E\to E$,
(\ref{detbd}) is equivalent to
$A^{-1}\circ (\id_E+v)\circ (A+h) = A^{-1} \circ (A+g)\circ (\id_E+v)$,
which in turn is equivalent to
\begin{equation}\label{vers1}
v \; =\;  A^{-1}\circ h +A^{-1}\circ v\circ (A+h)-A^{-1}\circ g\circ (\id_E+v)\, .
\end{equation}
Let
$\pi_\obs\colon E\to E_\obs$
and $\pi_\obu\colon E\to E_\obu$
be the projections onto the stable and unstable subspace of $E$, respectively.
In the following, we identify a function $k\colon E\to E$ with the pair $(k_\obs,k_\obu)$
of its components $k_\obs:=\pi_\obs\circ k$
and $k_\obu:=\pi_\obu\circ k$.
Then $BC(E,E)=BC(E,E_{\obs})\oplus BC(E,E_{\obu})$
as a Banach space (if we take the maximum norm
on the right hand side).
If $v=(v_\obs,v_\obu)$, then (\ref{vers1}) holds if and only if
both (\ref{vers2}) and (\ref{vers3}) are satisfied:
\begin{eqnarray}
v_\obs &\!=\!& A^{-1}_1\circ h_\obs +A^{-1}_1\circ v_\obs \circ (A+h)-A^{-1}_1 \circ g_\obs \circ (\id_E+v)\label{vers2}\\
v_\obu &\!=\! & A^{-1}_2 \circ h_\obu +A^{-1}_2\circ v_\obu\circ (A+h)-A^{-1}_2 \circ g_\obu\circ (\id_E+v)=:\theta_2(v).\label{vers3}
\end{eqnarray}
Moreover, (\ref{vers2}) is satisfied if and only if
\begin{equation}\label{vers4}
v_\obs \,=\, A_1\circ v_\obs\circ (A+h)^{-1}
 -h_\obs\circ (A+h)^{-1}+g_\obs\circ (\id_E+v)\circ (A+h)^{-1}\,=:\, \theta_1(v).
\end{equation}
Thus (\ref{detbd}) holds if and only if $v\in BC(E,E)$ is a fixed point of the self-map
\[
\theta:=(\theta_1,\theta_2)\colon BC(E,E)\to
BC(E,E_{\obs})\oplus BC(E,E_{\obu})=BC(E,E)
\]
of the Banach space $BC(E,E)$.
\emph{We claim that $\theta$ is a contraction, with $\Lip(\theta)\leq\Lambda$}.
If this is true, then
$\theta$ will have a unique fixed
point (the unique $v$ we seek),
by Banach's Fixed Point Theorem \cite[Theorem~3.4.1]{KaP}.
Starting the iterative approximation of $v$ with the zero-function
$v_0:=0\colon E\to E$, the standard \emph{a priori} estimate
(see \cite[Proposition~3.4.4]{KaP}) gives
\[
\|v\|_\infty=\|v-v_0\|_\infty \leq \frac{\|\theta(v_0)-v_0\|_\infty}{1-\Lambda}
=\frac{\|\theta(v_0)\|_\infty}{1-\Lambda}\,
\]
and applying now the triangle inequality to the individual summands in $\|\theta(v_0)\|_\infty=\|\theta(0)\|_\infty
=\max\{\|\theta_1(0)\|_\infty\|\theta_2(0)\|_\infty\}$ (as in (\ref{vers4}) and (\ref{vers3})),
we obtain (\ref{unifestim}).
If $g(0)=h(0)=0$,
we have $\theta^n(v_0)(0)=0$ for each $n\in \N_0$, by a trivial induction.
Hence also $v=\lim_{n\to\infty}\theta^n(v_0)$ vanishes at~$0$.\\[2.5mm]
%
%
%=================================
%AB HIER NOCHMAL KORREKTURLESEN
%=================================
%
%
To establish the claim,
we need only show that both $\Lip(\theta_1),\Lip(\theta_2)\leq \Lambda$,
because $\Lip(\theta)=\max\{\Lip(\theta_1),\Lip(\theta_2)\}$.
Given $v,w\in BC(E,E)$,
we have
\begin{eqnarray*}
\|\theta_2(v)-\theta_2(w)\|_\infty & \leq &
\|A^{-1}_2\circ (v_\obu-w_\obu) \circ (A+h)\|_\infty\\
& & \!\!\!\!  + \,
\|A^{-1}_2 \circ g_\obu\circ (\id_E+v)-
A^{-1}_2 \circ g_\obu\circ (\id_E+w)\|_\infty \,.
\end{eqnarray*}
Since
$\|A^{-1}_2\circ (v_\obu -w_\obu ) \circ (A+h)\|_\infty \leq
\|A^{-1}_2\| \cdot \|v-w\|_\infty$
and
\[
\|A^{-1}_2 \circ g_\obu\circ (\id_E+v)-
A^{-1}_2 \circ g_\obu\circ (\id_E+w)\|_\infty
\leq
\|A^{-1}_2 \| \cdot \Lip(g_\obu) \circ \|v-w\|_\infty\,,
\]
we get
$\|\theta_2(v)-\theta_2(w)\|_\infty\leq
\|A^{-1}_2\| (1 +\Lip(g_\obu))  \|v-w\|_\infty$
and thus
\begin{equation}\label{esthet1}
\Lip(\theta_2)\leq \|A^{-1}_2\| (1 +\Lip(g_\obu))  \; \leq \; \Lambda
\end{equation}
(using (\ref{hyp2})).
Moreover,
\[
\|\theta_1(v)-\theta_1(w)\|_\infty
\; \leq\hspace*{110mm}\vspace{-1.5mm}
\]
\[
 \|A_1\circ (v_\obs -w_\obs )\|_\infty
 + 
 \| g_\obs\circ (\id_E+v)\circ (A+h)^{-1}
-g_\obs\circ (\id_E+w)\circ (A+h)^{-1}\|_\infty.\vspace{1.5mm}
\]
As
$\| g_\obs\circ (\id_E+v)\circ (A+h)^{-1}
-g_\obs\circ (\id_E+w)\circ (A+h)^{-1}\|_\infty\leq \Lip(g_\obs)\|v-w\|_\infty$
and $\|A_1\circ (v_\obs-w_\obs )\|_\infty\leq \|A_1\|\cdot \|v-w\|_\infty$,
we obtain
\begin{equation}\label{esthetXX}
\Lip(\theta_1)\leq \|A_1\| +\Lip(g_\obs)  \; \leq\; \Lambda
\end{equation}
(using (\ref{hyp2}) again).
Thus
\begin{equation}\label{reusend}
\Lip(\theta)\;\leq\; \Lambda\, ,
\end{equation}
which completes the proof.
\end{proof}
\begin{la}\label{bdhomeo}
In the situation of Lemma~{\rm\ref{unqbdd}},
assume that also
\begin{equation}\label{hyp3}
\Lip(g) <  \|A^{-1}\|^{-1}\, ,
\end{equation}
\begin{equation}\label{hyp4}
\|A^{-1}_2\| (1 +\Lip(h_\obu))  \; < \; 1\, ,\quad \mbox{and}\quad
\|A_1\| +\Lip(h_\obs)  \; < \; 1
\end{equation}
hold.
Then the map $\id_E+v\colon E\to E$
is a homeomorphism.
Moreover, $w:=(\id_E+v)^{-1}-\id_E\colon E\to E$
is the unique bounded continuous
map such that
\begin{equation}\label{detbd2}
(\id_E+w)\circ (A+g)\;=\; (A+h)\circ (\id_E+w)\,.
\end{equation}
\end{la}
\begin{proof}
In view of (\ref{hyp3}) and (\ref{hyp4}),
we can apply Lemma~\ref{unqbdd} with reversed roles
of $g$ and $h$, and obtain a unique bounded continuous
map $w \colon E\to E$ such that (\ref{detbd2})
holds.
Then
\[
(\id_E+v)\circ (\id_E+w)\;=\; \id_E+f,
\]
where $f:=w+v\circ (\id_E+w)$ is continuous and bounded.
Now
\begin{eqnarray*}
(\id_E+f)\circ (A+g)
&\!=\! &
(\id_E+v)\circ (\id_E+w)\circ (A+g)\\
& \! =\! & 
(\id_E+v)\circ (A+h)\circ (\id_E+w)\\
& \!= \! & (A+g)\circ (\id_E+v)\circ (\id_E+w)
= (A+g)\circ (\id_E+f),
\end{eqnarray*}
using (\ref{detbd2}) to obtain the second equality and
(\ref{detbd}) for the third.
Since also $(\id_E+0)\circ (A+g)=(A+g)\circ (\id_E+0)$,
the uniqueness property in Lemma~\ref{unqbdd}
(applied to $g$ and $g$ in place of $g$ and $h$)
shows that $f=0$
and therefore\linebreak
$(\id_E+v)\circ (\id_E+w)=\id_E$.
Reversing the roles of $g$ and $h$,
the same argument gives
$(\id_E+w)\circ (\id_E+v)=\id_E$.
Thus $\id_E+v$ is invertible
with $(\id_E+v)^{-1}=\id_E+w$.
The assertions follow.
\end{proof}
\section{H\"{o}lder property of the conjugacies}
We now show that the mappings $v$ constructed in Section \ref{bforth}
are H\"older.
\begin{la}\label{limithoel}
Let $(X,d_X)$ and $(Y,d_Y)$ be metric spaces,
$\alpha>0$
and $(f_j)_{j\in J}$
be a net in $L_\alpha(X,Y)$
which converges pointwise to a function $f\colon X\to Y$.
If
\[
\lambda\; :=\; \sup\{\Lip_\alpha(f_j) \colon j\in J\}\;<\;\infty\,,
\]
then $f\in L_\alpha(X,Y)$ and $\Lip_\alpha(f)\leq\lambda$.
\end{la}
\begin{proof}
Given $x,y\in X$, we have
$d_Y(f_j(x),f_j(y))\leq \lambda\, d_X(x,y)^\alpha$
for all $j\in J$.
Passing to the limit, we obtain $d_Y(f(x),f(y))\leq\lambda \, d_X(x,y)^\alpha$.
\end{proof}
\begin{la}\label{gethoelder}
In the situation of Lemma~{\rm\ref{unqbdd}},
let $k:=(A+h)^{-1}$ and assume that
\begin{equation}\label{exneed1}
\Lip_\alpha(h_\obs\circ k)
+
\Lip(k)^\alpha
\bigl(\ve \|A_1\|+
\max\{\Lip(g_\obs)(1+\ve),\spread(g_\obs)\}\bigr) \leq \ve
\end{equation}
and
\begin{eqnarray}
\lefteqn{\Lip_\alpha(A^{-1}_2\circ h_\obu)
\,+ \, \ve \|A_2^{-1}\|\Lip(A+h)^\alpha}\hspace*{18mm}\notag\\
& & +\;  \|A^{-1}_2\| \max\{\Lip(g_\obu)(1+\ve),\spread(g_\obu)\}
\; \leq  \; \ve \qquad\qquad \label{exneed2}
\end{eqnarray}
for a given number $\alpha\in \;]0,1[$.
Then the bounded continuous map~$v\colon E\to E$ determined
by {\rm(\ref{detbd})}
is H\"{o}lder of exponent~$\alpha$,
and
\[
\Lip_\alpha(v)\; \leq\; \ve\, .
\]
\end{la}
\begin{proof}
We retain the notation
introduced in the proof of Lemma~\ref{unqbdd};
in particular, we shall use the contraction
$\theta=(\theta_1,\theta_2)\colon BC(E,E)\to BC(E,E)$
introduced there.
By Lemma~\ref{limithoel},
the (non-empty) set
\begin{equation}\label{defnY}
Y\; :=\; \{f\in BC(E,E)\cap L_\alpha(E,E)\colon \Lip_\alpha(f)\leq \ve\}
\end{equation}
is closed in $BC(E,E)$, and hence a complete metric
space with the metric induced by that on $BC(E,E)$,
$d_\infty (u,w):=\|u-w\|_\infty$.
We claim that $\theta(Y)\sub Y$.
If this is true, then the Banach Fixed Point Theorem
provides a unique fixed point
$y\in Y$ for the contraction $\theta|_Y\colon Y\to Y$
of $(Y,d_\infty)$.Then $y$ has to coincide with
the unique fixed point $v\in BC(E,E)$ of $\theta$
(the map $v$ determined by (\ref{detbd})), and thus $v=y\in Y$,
whence all assertions of the lemma hold.
Since
\begin{eqnarray}
L_\alpha(E,E)& =& L_\alpha(E,E_\obs)\oplus L_\alpha(E,E_\obu)\quad\mbox{and}\label{ismaxn}\\
\Lip_\alpha(f) & = & \max\{\Lip_\alpha(f_\obs),\Lip_\alpha(f_\obu)\}\label{ismaxn2}
\end{eqnarray}
for $f=(f_\obs,f_\obu)\in L_\alpha(E,E)$,
to establish the claim we need only show that
both $\theta_1(v)$ and $\theta_2(v)$ are H\"{o}lder
of exponent~$\alpha$ for each $v\in Y$,
and $\Lip_\alpha(\theta_1(v)),\Lip_\alpha(\theta_2(v))\leq\ve$.
In view of Lemmas~\ref{basecom}, \ref{desc}
and \ref{compohoeld},
all three summands in (\ref{vers4})
are H\"{o}lder of exponent~$\alpha$.
Now
\begin{equation}\label{combi1}
\Lip_\alpha(\theta_1(v))\, \leq\,
\Lip_\alpha(A_1\circ v_\obs\circ k)+
\Lip_\alpha(h_\obs\circ k)
+\Lip_\alpha(g_\obs\circ (\id_E+v)\circ k)\, ,
\end{equation}
where
$\Lip_\alpha(A_1\circ v_\obs\circ k)\leq \|A_1\|\Lip_\alpha(v_\obs)\Lip(k)^\alpha
\leq \ve \|A_1\| \Lip(k)^\alpha$
by Lemma~\ref{basecom}
and
\begin{eqnarray*}
\Lip_\alpha(g_\obs\circ (\id_E+v)\circ k) & \leq &
\Lip_\alpha(g_\obs\circ(\id_E+v))\Lip(k)^\alpha\\
&\leq&  \max\{\Lip(g_\obs)(1+\Lip_\alpha(v)),\spread(g_\obs)\}\Lip(k)^\alpha\\
&\leq&  \max\{\Lip(g_\obs)(1+\ve),\spread(g_\obs)\}\Lip(k)^\alpha
\end{eqnarray*}
by Lemmas~\ref{basecom}
and \ref{compohoeld}.
To obtain an upper bound for $\Lip_\alpha(\theta_1(v))$,
we substitute the preceding estimates
into (\ref{combi1}). The upper bound so obtained is the left hand side
of (\ref{exneed1}) and hence $\leq \ve$ by hypotheses.
Thus $\Lip_\alpha(\theta_1(v))\leq\ve$.
Similarly, Lemmas~\ref{basecom}, \ref{desc}
and \ref{compohoeld} show that
all three summands in (\ref{vers3})
are H\"{o}lder of exponent~$\alpha$.
Now
\begin{eqnarray}
\Lip_\alpha(\theta_2(v))& \! \leq\! &
\Lip_\alpha(A^{-1}_2 \circ h_\obu)
+
\Lip_\alpha(A^{-1}_2\circ v_\obu\circ (A+h)) \notag \\
& &\qquad +\; \Lip_\alpha(A^{-1}_2 \circ g_\obu\circ (\id_E+v))\, ; \label{combi2}
\end{eqnarray}
here
$\Lip_\alpha\hspace*{-.3mm}(A^{-1}_2
\hspace*{-.1mm}\circ
\hspace*{-.1mm} v_\obu
\hspace*{-.1mm}\circ\hspace*{-.1mm}
 (A\hspace*{-.2mm}+\hspace*{-.2mm}h))
\hspace*{-.4mm}\leq
\hspace*{-.4mm}
 \|A_2^{-1}\|\hspace*{-.2mm}\Lip_\alpha(v_\obu)\Lip(A\hspace*{-.2mm}+\hspace*{-.2mm}h)^\alpha
\hspace*{-.4mm}\leq\hspace*{-.4mm} \ve \|A_2^{-1}\|\hspace*{-.2mm}\Lip(A\hspace*{-.2mm}+\hspace*{-.2mm}h)^\alpha$
by Lemma~\ref{basecom}
and
\begin{eqnarray*}
\Lip_\alpha(A^{-1}_2\circ g_\obu\circ (\id_E+v))
&\!\! \leq \!\! & \|A^{-1}_2\|
\max\{\Lip(g_\obu)(1+\Lip_\alpha(v)),\spread(g_\obu)\}\\
&\!\! \leq\!\! &
\|A^{-1}_2\| \max\{\Lip(g_\obu)(1+\ve),\spread(g_\obu)\}
\end{eqnarray*}
by Lemmas~\ref{basecom} and \ref{compohoeld}.
Combining (\ref{combi2}) with the preceding estimates,
we get the left hand side of (\ref{exneed2})
as an upper bound for $\Lip_\alpha(\theta_2(v))$.
Hence also $\Lip_\alpha(\theta_2(v))\leq \ve$
and thus $\theta(v)\in Y$,
which completes the proof.
\end{proof}
The conditions (\ref{exneed1}) and (\ref{exneed2})
describe exactly what we need
in the proof, but they are somewhat elusive.
They can be replaced by
stronger (but more tangible) hypotheses,
which we now state.
\begin{la}\label{lazyla}
If $g$ and $h$ are as in Lemma~{\rm\ref{unqbdd}}
and
\begin{eqnarray}
\lefteqn{\frac{\ve \|A_1\|}{(\|A^{-1}\|^{-1}-\Lip(h))^\alpha}
+\max\Big\{\frac{\Lip(h_\obs)}{\|A^{-1}\|^{-1}-\Lip(h)},\spread(h_\obs)\Big\}}\qquad\qquad\qquad
\notag \\[1.5mm]
& & +\, \frac{\max\{\Lip(g_\obs)(1+\ve),\spread(g_\obs)\}}{(\|A^{-1}\|^{-1}-\Lip(h))^\alpha} \; \leq \; \ve\label{lazy1}
\end{eqnarray}
as well as
\begin{eqnarray}
\lefteqn{\|A^{-1}_2\|
\max\{\Lip(h_\obu),\spread(h_\obu)\}
+ \ve \|A_2^{-1}\| (\|A\|+\Lip(h))^\alpha}\hspace*{18mm}\notag \\[1.7mm]
& & \quad \quad +   \, \|A^{-1}_2\| \max\{\Lip(g_\obu)(1+\ve),\spread(g_\obu)\} \leq \ve
,\label{lazy2}
\end{eqnarray}
then the conditions {\rm(\ref{exneed1})} and {\rm(\ref{exneed2})}
from Lemma~{\rm\ref{gethoelder}}
are satisfied.
In particular,
if $\alpha\in\;]0,1[$ and $\ve>0$ are given and we choose $\delta>0$ so small that
\begin{equation}\label{lazy3}
\delta <  \|A^{-1}\|^{-1}\,\quad
\|A^{-1}_2\| (1 +\delta)   < 1,\quad
\|A_1\| +\delta   <  1,
\end{equation}
\begin{equation}\label{lazy4}
2\|A^{-1}_2\|\delta
+ \ve \|A_2^{-1}\| (\|A\|+\delta)^\alpha
+   \|A^{-1}_2\| \max\{\delta(1+\ve),2\delta\}\leq \ve, \;\;\, \mbox{and }\;
\end{equation}
\begin{equation}\label{lazy5}
\frac{\ve \|A_1\|}{(\|A^{-1}\|^{-1}-\delta)^\alpha}
+\max\Big\{\frac{\delta}{\|A^{-1}\|^{-1}-\delta},2\delta\Big\}
+ \frac{\max\{\delta(1+\ve),2\delta\}}{(\|A^{-1}\|^{-1}-\delta)^\alpha} \; \leq \; \ve,
\end{equation}
then
conditions {\rm(\ref{hyp1})},
{\rm(\ref{hyp2})},
{\rm(\ref{exneed1})}
and
{\rm(\ref{exneed2})}
are satisfied
for all bounded, Lip\-schitz maps $g, h\colon E\to E$ with
\begin{equation}\label{lazlaz}
\max\{\|g\|_\infty,\Lip(g)\}\, \leq\,  \delta
\quad\mbox{ and }\quad
\max\{\|h\|_\infty,\Lip(h)\}\, \leq\,  \delta\,.
\end{equation}
\end{la}
\begin{proof}
Let $k:=(A+h)^{-1}$, as in Lemma~\ref{gethoelder}. Then
\begin{equation}\label{stpp1}
\Lip(k)\,\leq\,
\frac{1}{\|A^{-1}\|^{-1}-\Lip(h)}\,,
\end{equation}
by (\ref{lipforinv}).
Next,
\begin{eqnarray}
\Lip_\alpha(h_\obs \circ k)
& \leq & 
\max\{\Lip(h_\obs\circ k),\spread(h_\obs\circ k)\}\notag \\
& \leq &
\max\Big\{\frac{\Lip(h_\obs)}{\|A^{-1}\|^{-1}-\Lip(h)},
\spread(h_\obs)\Big\},\label{stpp2}
\end{eqnarray}
using Lemma~\ref{desc},
Lemma~\ref{basecom}, and the estimate (\ref{stpp1}).
We also have
\begin{equation}\label{stpp3}
\Lip_\alpha(A^{-1}_2\circ h_\obu)
\leq
\|A^{-1}_2 \|  \Lip_\alpha(h_\obu)
\leq 
\|A^{-1}_2\| \max\{  \Lip(h_\obu),\spread(h_\obu)\},
\end{equation}
using Lemmas~\ref{basecom} and \ref{desc}.
Finally, we have
\begin{equation}\label{stpp4}
\ve \|A_2^{-1}\|\Lip(A+h)^\alpha
\leq
\ve \|A_2^{-1}\| (\|A\| + \Lip(h))^\alpha\,.
\end{equation}
In view of (\ref{stpp1})--(\ref{stpp4}),
it is clear that (\ref{lazy1})
implies (\ref{exneed1})
and (\ref{lazy2}) implies (\ref{exneed2}).
The final assertion of the lemma is now obvious,
using that $\spread(f)\leq 2\|f\|_\infty$
for all bounded maps $f$ between normed spaces.
\end{proof}
\begin{rem}\label{remtherem}
\begin{itemize}
\item[(a)]
Note that, given $h,g$ as in Lemma~\ref{unqbdd}, one can always find $\alpha\in \;]0,1[$
and $\ve>0$ such that (\ref{lazy1}) and (\ref{lazy2})
(and hence also (\ref{exneed1})
and (\ref{exneed2})) are satisfied.
In fact, we have $1-\|A_1\|-\Lip(g_\obs)>0$ by (\ref{hyp2})
and hence also
\begin{equation}\label{lill}
\Delta_{g,h}:=1-\frac{\|A_1\|+\Lip(g_\obs)}{(\|A^{-1}\|^{-1}-\Lip(h))^\alpha}>0
\end{equation}
for sufficiently small $\alpha\in \;]0,1[$.
Instead of (\ref{lazy1}), to simplify the calculation
let us impose a stronger
condition by replacing the second maximum $\max\{\Lip(g_s)(1+\ve),\spread(g_s)\}$
in (\ref{lazy1})
by the larger term
\[
\max\{\Lip(g_\obs),\spread(g_\obs)\} +\ve\Lip(g_\obs)\,.
\]
We can then solve for $\ve$ and see
that the strengthened inequality is equivalent to
\begin{equation}\label{lazy7}
\hspace*{-1.5mm}\ve \geq \frac{\max\Big\{\frac{\Lip(h_\obs)}{\|A^{-1}\|^{-1}-\Lip(h)},\spread(h_\obs)\Big\}+
\frac{\max\{\Lip(g_\obs),\spread(g_\obs)\}}{(\|A^{-1}\|^{-1}-\Lip(h))^\alpha}}{\Delta_{g,h}}\!.\!
\end{equation}
Also, we have $1-\|A_2^{-1}\|(1+\Lip(g_\obu))>0$ by (\ref{hyp2})
and hence
\begin{equation}\label{lill2}
\delta_{g,h}:=
1-\|A_2^{-1}\|((\|A\|+\Lip(h))^\alpha+ \Lip(g_\obu)) >0
\end{equation}
for sufficiently small $\alpha\in \;]0,1[$.
Likewise, replacing $\|A^{-1}_2\|$ times the
second maximum in (\ref{lazy2})
by
\[
\|A^{-1}_2\| \max\{\Lip(g_\obu),\spread(g_\obu)\} +\ve \|A_2^{-1}\|\Lip(g_\obu)\,,
\]
we obtain a stronger condition equivalent
to
\begin{equation}\label{lazy8}
\hspace*{-3.1mm}\ve \geq \frac{\|A^{-1}_2\|(\max\{\Lip(h_\obu),\spread(h_\obu)\}
\!+\!   \max\{\Lip(g_\obu),\spread(g_\obu)\})}{\delta_{g,h}}\!.\!
\end{equation}
Now choose $\ve$ so large that both (\ref{lazy7})
and (\ref{lazy8}) hold.
\item[(b)]
Given $g$ and $h$ as in Lemma~\ref{unqbdd}, we can actually find $\alpha\in \;]0,1[$
and $\ve>0$ such that (\ref{lazy1}) and (\ref{lazy2})
are satisfied simultaneously for $(g,h)$ and $(h,g)$
(i.e., with reversed roles of $h$ and $g$): Simply proceed
as in (a) for both pairs, and replace the values of $\alpha$ obtained by their minimum.
Then choose an $\ve$ 
for this $\alpha$ in both cases, and replace the two values of $\ve$
by their maximum.
\item[(c)]
Note that we did not need to assume that $g(0)=0$
or $h(0)=0$ in our previous results (although, of course,
this case is of primary interest).
\item[(d)]
Because $\spread(f)\leq 2\|f\|_\infty$,
one can replace $\spread(f)$
with $2\|f\|_\infty$ in (\ref{lazy1})
and (\ref{lazy2}) for $f=g_\obs,g_\obu,h_\obs,h_\obu$,
and obtains simpler-looking, alternative conditions 
which also imply (\ref{exneed1}) and (\ref{exneed2}).
\end{itemize}
\end{rem}
\section{Proof of Theorem A}
The assertions of the theorem are
covered by Lemmas~\ref{unqbdd}, \ref{bdhomeo} and \ref{gethoelder}
and Remark~\ref{remtherem}\,(a), setting $h:=0$ there.
\section{Proof of Theorem B}
We give the proof in a form which can be re-used later in the study of
parameter dependence. Avoiding only a trivial case, assume $E\not=\{0\}$.
After a translation, we may (and will) assume that $x=0$.
After shrinking $P$, we may also assume that $P=B^E_r(0)$ for some $r>0$.
Write $f(y)=f(0)+f'(0)(y)+R(y)$; thus
\[
f(y)=A(y)+R(y)\qquad \mbox{for all $y\in B_r^E(0)$,}
\]
with $A:=f'(0)$. Let $E=E_\obs\oplus E_\obu$ with respect to $A$ and
$\|.\|$ be an adapted norm on~$E$.
\begin{numba}
If $\K$ and $E$ are ultrametric, then also the adapted norm $\|.\|$ on $E$ can (and will)
be chosen ultrametric (see Appendix~\ref{fiappe}). In this case, we define
$R_s\colon E \to E$ for $s\in \;]0,r]$ via
\begin{equation}\label{LNew1}
R_s(y):=\left\{
\begin{array}{cl}
R(y) & \mbox{if $\,y\in B^E_s(0)$;}\\
0 & \mbox{else.}
\end{array}
\right.
\end{equation}
Choose $s$ so small that $R|_{B^E_s(0)}$ is Lipschitz (see (\ref{stricvialip})).
If $y,z\in B_s^E(0)$, then $\|R_s(z)-R_s(y)\|=\|R(z)-R(y)\|\leq\Lip(R|_{B^E_s(0)})\|z-y\|$.
If $y,z\in E\setminus B^E_s(0)$, then $\|R_s(z)-R_s(y)\|=0$.
If $z\in B^E_s(0)$ and $y\in E\setminus B^E_s(0)$, then $\|z-y\|=\|y\|>\|z\|$ by (\ref{stronger})
and thus
$\|R_s(z)-R_s(y)\|=\|R(z)\|=\|R(z)-R(0)\|\leq\Lip(R|_{B^E_s(0)})\|z\|
\leq \Lip(R|_{B^E_s(0)})\|z-y\|$. Hence $R_s$ is Lipschitz, with
\begin{equation}
\Lip(R_s)\leq \Lip(R|_{B^E_s(0)})
\end{equation}
(and in fact equality holds).
\end{numba}
\begin{numba}\label{introeta}
In the real case, let $\eta\colon [0,\infty[\to [0,1]$
be a Lipschitz function (with respect to the ordinary absolute value on $\R$)
such that $\eta|_{[0,1]}=1$
and $\eta(t)=0$ for $t\geq 2$. Then
\begin{equation}\label{esteta}
\Lip(\eta)\geq 1\,.
\end{equation}
For $s\in \;]0,r/3]$, define
\[
\xi_s\colon E\to [0,1],\quad \xi_s(y):=\eta(\|y\|/s)
\]
and
\begin{equation}\label{LNew2}
R_s(y):=\left\{
\begin{array}{cl}
\xi_s(y)R(y) & \mbox{if $\,y\in B^E_{3s}(0)$;}\\
0 & \mbox{else.}
\end{array}
\right.
\end{equation}
Choose $s$ so small that $R|_{B^E_{3s}(0)}$ is Lipschitz.
Then
\begin{equation}\label{estliprs}
\Lip(R_s)\leq (1+3\Lip(\eta))\Lip(R|_{B^E_{3s}(0)}),
\end{equation}
by the following arguments. First,
\begin{eqnarray*}
\Lip(R_s|_{B^E_{3s}(0)}) & \leq &
\Lip(\xi_s)\, \|R|_{B^E_{3s}(0)}\|_\infty+\|\xi_s\|_\infty \Lip(R|_{B^E_{3s}(0)})\\
&\leq &
\frac{1}{s}\Lip(\eta)\, 3s\Lip(R|_{B^E_{3s}(0)})+\Lip(R|_{B^E_{3s}(0)})\\
&=& (1+3\Lip(\eta))\Lip(R|_{B^E_{3s}(0)})
\end{eqnarray*}
(using Lemma~\ref{pwplip}
for the first inequality).
If $y\in E\setminus B^E_{3s}(0)$ and $z\in E$,
then $\|R_s(z)-R_s(y)\|\not=0$ implies $z\in B^E_{2s}(0)$.
In this case, $\|z-y\|\geq s$ and therefore
$\|R_s(z)-R_s(y)\|=\|R_s(z)\|\leq \|R(z)\|\leq \Lip(R|_{B^E_{3s}(0)})\|z\|
\leq \Lip(R|_{B^E_{3s}(0)})2s
\leq \Lip(R|_{B^E_{3s}(0)})2\|z-y\|
\leq (1+3\Lip(\eta))\Lip(R|_{B^E_{3s}(0)})\|z-y\|$.
\end{numba}
\begin{numba}\label{chod}
Returning to general $\K$,
given arbitrary $\alpha\in\;]0,1[$ and $\ve>0$
we choose $\delta>0$ so small that (\ref{lazy3}), (\ref{lazy4})
and (\ref{lazy5}) are satisfied.
\end{numba}
\begin{numba}\label{chos1}
In the ultrametric case,
we use (\ref{stricvialip}) to find $s\in\;]0,r]$ such that
\begin{equation}\label{condpr2}
\Lip(R|_{B^E_s(0)})\leq \delta
\end{equation}
and $s\leq 1$. Then $\|R_s(y)\|\leq \Lip(R_s(y))\|y\|\leq \delta s\leq \delta$
whenever $\|R_s(y)\|\not=0$, and hence
\begin{equation}\label{getth}
\|R_s\|_\infty\leq \delta.
\end{equation}
\end{numba}
\begin{numba}\label{chos2}
In the real case,
(\ref{stricvialip}) provides $s\in\;]0,r]$ such that
\begin{equation}\label{condpr3}
\Lip(R|_{B^E_{3s}(0)})\leq \frac{\delta}{1+3\Lip(\eta)}
\end{equation}
and $3s\leq 1$. Then again (\ref{getth}) holds.
\end{numba}
\begin{numba}\label{LNew3}
Now set $g:=R_s$ as just selected, and $h:=0$.
Because $\Lip(g)\leq \delta$ by choice of $s$ and
$\|g\|_\infty\leq\delta$ by (\ref{getth}),
condition (\ref{lazlaz}) is satisfied.
Hence both $(g,h)$ and $(h,g)$ satisfy
the conditions {\rm(\ref{hyp1})},
{\rm(\ref{hyp2})},
{\rm(\ref{exneed1})}
and
{\rm(\ref{exneed2})},
by Lemma~\ref{lazyla}.
Hence there are unique $v,w\in BC(E,E)$
to which all conclusions of
Lemmas~\ref{unqbdd}, \ref{bdhomeo} and \ref{gethoelder}
apply.
In particular, $v,w\in BL_\alpha(E,E)$
with
\begin{equation}\label{LNew8}
\Lip_\alpha(v),\Lip_\alpha(w)\leq \ve\, ,
\end{equation}
and $\id_E+v$ is a homeomorphism
with inverse $\id_E+w$. Since $h(0)=0$ and $g(0)=R(0)=0$, we also have
$v(0)=0$ and $w(0)=0$.
\end{numba}
\begin{numba}\label{singl}
If we are only interested in a single given function $f$,
we can now complete the proof by setting
$V:=B^E_s(0)$,
$U:=(\id_E+v)^{-1}(B^E_s(0))$
and $H:=(\id_E+v)|_U\colon U\to V$.
Since
\begin{equation}\label{splcase}
(A+g)\circ (\id_E+v)\, =\, (\id_E+v)\circ A
\end{equation}
and $R|_V=g|_V$, we then have
\[
f\circ H=f|_V\circ H=(A+g)|_V\circ (\id_E+v)|_U=(\id_E+v)\circ A|_U\,,
\]
from which (\ref{goaleq}) follows. This completes the proof.
\end{numba}
\begin{numba}\label{lasttango}
Since our previous choice of $U$ depends on $v$
(and hence on $f$), it is unsuitable for the study of parameter dependence.
To enable the latter, we need to make a different (usually smaller) choice
of $U$, which we now describe.
It is helpful to observe that
\begin{equation}\label{thetau}
\omega\colon [0,\infty[\to[0,\infty[\,,\quad \omega(a):=a+\ve a^\alpha
\end{equation}
is a monotonically increasing bijection, such that
$\omega(a)\geq a$ (and hence $\omega^{-1}(a)\leq a$)
for all $a\geq 0$.
Now
\begin{equation}\label{continf}
(\id_E+v)^{-1}(B_t^E(0))\,\supseteq\, B_{\omega^{-1}(t)}^E(0)\quad\mbox{for all $\,t>0$.}
\end{equation}
In fact, given $a>0$,
we have $\|y+v(y)\|\leq \|v\|+\Lip_\alpha(v)\|y\|^\alpha\leq a+\ve a^\alpha=\omega(a)$
for each $y\in B^E_a(0)$, and thus
\[
(\id_E+v)(B^E_a(0))\, \sub \, B^E_{\omega(a)}(0)\, .
\]
Hence $B^E_a(0) \sub (\id_E+v)^{-1}(B^E_{\omega(a)}(0))$,
entailing (\ref{continf}) (with $a:=\omega^{-1}(t)$).\\[2.5mm]
We now set $U:=B^E_{\omega^{-1}(s)}(0)\sub B^E_s(0)$.
Since $V:=(\id_E+v)(U)\sub B^E_s(0)$
by the preceding discussion,
we can set $H:=(\id_E+v)|_U\colon U\to V$
and complete the discussion as in \ref{singl}.
\end{numba}
\section{Parameter dependence of the conjugacy}
Before we can study parameter dependence of the conjugacies constructed earlier,
we compile various auxiliary results.
The first lemma is probably part of the folklore.
See \cite[Theorem 21]{IRCOMP} for the Lipschitz case;
for completeness, the general proof is given in Appendix~\ref{nowaB}.
\begin{la}[H\"{o}lder dependence of fixed points on parameters]\label{hoedepla}
\hspace*{-1.5mm}Let\linebreak
$(X,d_X)$ and
$(Y,d_Y)$ be metric spaces,
$\alpha> 0$ and $f\colon X\times Y\to Y$
be a mapping with the following three properties:
\begin{itemize}
\item[\rm(a)]
The family $(f^y)_{y\in Y}$ of the maps
$f^y\colon X\to Y$, $f^y(x):=f(x,y)$
is uniformly H\"{o}lder of exponent $\alpha$,
in the sense that each $f^y$ is H\"{o}lder of exponent $\alpha$
and
\[
\mu\; :=\; \sup\{ \Lip_\alpha(f^y)\colon y\in Y\}\;<\; \infty\,.
\]
\item[\rm(b)]
The maps
$f_x\colon Y\to Y$, $y\mto f(x,y)$,
with $x\in X$, form a uniform family $(f_x)_{x\in X}$
of contractions,
in the sense that each $f_x$ is a contraction and
\[
\lambda \; :=\; \sup\{\Lip(f_x)\colon x\in X\}\;<\; 1\,.
\]
\item[\rm(c)]
For each $x\in X$, there exists a fixed point $y_x\in Y$
for $f_x$.
\end{itemize}
Then $y_x$ is uniquely determined and
the map
$\phi\colon X\to Y$, $\phi(x):= y_x$
is H\"{o}lder of exponent~$\alpha$, with
\[
\Lip_\alpha(\phi)\; \leq \; \frac{\mu}{1-\lambda} \,.
\]
\end{la}
\begin{rem}\label{userm}
Note that condition (a) of Lemma~\ref{hoedepla}
is satisfied in particular if
$f$ is H\"{o}lder of exponent $\alpha$
with respect to some metric $d$ on $X\times Y$
such that $d((x_1,y),(x_2,y))=d_X(x_1,x_2)$
for all $x_1,x_2\in X$ and $y\in Y$.
Condition (c) is satisfied whenever the metric space $(Y,d_Y)$
is complete (and $Y\not=\emptyset$),
by Banach's Fixed Point Theorem.
\end{rem}
The dependence of $w$ on $v$ in the situation of Lemma~\ref{globinv}
is considered next.
\begin{la}\label{globinvpar}
Let $(E,\|.\|)$ be a Banach space over a valued field $(\K,|.|)$
$($such that $E\not=\{0\})$,
and $0<\lambda <1$.
Let $A\colon E\to E$ be
an automorphism of topological vector spaces,
and $\Omega$ be the set of
all bounded, Lipschitz maps
$v\colon E\to E$ such that
\[
\Lip(v)\|A^{-1}\|\, \leq \, \lambda\,.
\]
Equip $\Omega$ with the metric given by
$d_\infty(v_1,v_2):=\|v_1-v_2\|_\infty$.
Given $v\in \Omega$, let
\[
w_v:=(A+v)^{-1}-A^{-1} \,.
\]
Then the map $\phi\colon \Omega\to BC(E,E)$, $v\mto w_v$
is Lipschitz, with
\[
\Lip(\phi)\, \leq \, \frac{\|A^{-1}\|}{1-\lambda}\, .
\]
\end{la}
\begin{proof}
Consider the map
\[
h\colon \Omega \times BC(E,E)\to BC(E,E)\,,\quad
h(v,u):=-A^{-1}\circ v\circ (A^{-1}+u)\, .
\]
We know from (\ref{enabiter})
that $w_v$ satisfies
\[
w_v\, =\, -A^{-1}\circ v\circ (A^{-1}+w_v)\, .
\]
Thus $w_v$ is a fixed point of $h_v:=h(v,.)$,
and it only remains to verify the hypotheses
of Lemma~\ref{hoedepla} for~$h$,
with $\mu\leq \|A^{-1}\|$ and the given~$\lambda$.
Each $h_v$ is Lipschitz, with
$\Lip(h_v)\leq \|A^{-1}\|\Lip(v)\leq \lambda$.
Hence $(h_v)_{v\in \Omega}$ is a uniform family of contractions.
Fix $u\in BC(E,E)$.
Given $v_1,v_2\in \Omega$,
we have
\[
\|h(v_2,u)-h(v_1,u)\|_\infty
=
\|A^{-1}\circ (v_2 -v_1)\circ (A^{-1}+u)\|_\infty\leq \|A^{-1}\|\,\|v_2-v_1\|_\infty\,.
\]
Hence $h(.,u)\colon \Omega\to BC(E,E)$ is Lipschitz with
$\Lip(h(.,u))\leq \|A^{-1}\|$, which completes the proof.
\end{proof}
A linear map $A\colon E\to F$
between Banach spaces over a locally compact,
valued field is called a \emph{compact operator}
if $A(B)$ is relatively compact in~$F$
for each bounded subset $B\sub E$
(or equivalently, if $A(B^E_1(0))\sub F$
is relatively compact). Then $A$ is continuous.
As it is similar to the classical real case,
we relegate the proof of the next result to the appendix
(Appendix~\ref{appnowC}).
\begin{la}\label{compop}
Let $(K,d)$ be a compact metric space,
$(E,\|.\|)$ be a normed space
over a valued field $(\K,|.|)$,
and $\alpha > \beta>0$.
Then $L_\alpha(K,E)\sub L_\beta(K,E)$.
Assume that, moreover,
$\K$ is locally compact and $E$
of finite dimension. If $|.|$ is ultrametric,
assume also that $d$ is ultrametric.
Then the inclusion map
\[
j_{\beta,\alpha}\colon L_\alpha(K,E)\to L_\beta(K,E)\,,\quad f\mto f
\]
is a compact operator.
\end{la}
\begin{la}\label{easgtinhoel}
Let $(K,d)$ be a compact metric space and $X\sub K$ be a dense subset.
Let $(E,\|.\|)$ be a finite-dimensional normed space
over a valued field $(\K,|.|)$ that is locally compact,
and $\alpha> \beta>0$. If $|.|$ is ultrametric,
assume that also $d$ is ultrametric.
Let $B\sub BL_\alpha(X,E)$ be bounded; thus
\[
\sup_{f\in B}\|f\|_\infty\,<\,\infty\quad\mbox{and}\quad
\sup_{f\in B}\Lip_\alpha(f)\,<\, \infty\,.
\]
Then $BC(X,E)$ and $BL_\beta(X,E)$ induce the same
topology on $B$.
\end{la}
\begin{proof}
Assume first that $X=K$.
By Lemma~\ref{compop}, the closure $\wb{B}\sub L_\beta(K,E)$ is compact.
Because the topology on $\wb{B}$ induced by $C(K,E)$ is Hausdorff and coarser
than the previous compact topology, the two topologies coincide.
The same then holds for the topologies on the smaller set~$B$.
In the general case,
each $f\in BL_\alpha(X,E)$ extends (by uniform continuity)
uniquely to a continuous function $\wt{f}\colon K \to E$.
Then $\Lip_\alpha(f)=\Lip_\alpha(\wt{f})$
(as we can pass to limits in $(x,y)$ in the H\"{o}lder condition),
and thus $BL_\alpha(X,E)\to BL_\alpha(K,E)$, $f\mapsto \widetilde{f}$
is an isometric isomorphism. Likewise with $\beta$ in place of $\alpha$.
The assertion hence follows from the result for maps on~$K$, as just proved.
\end{proof}
\begin{thm}\label{globparthm}
Let $E$ be a Banach space over a valued field $(\K,|.|)$
and $d_\infty\colon BC(E,E)^2\to[0,\infty[$, $d_\infty(h_1,h_2):=\|h_1-h_2\|_\infty$ be the supremum metric.
Let $A\colon E\to E$ be a hyperbolic automorphism
and $\alpha\in \,]0,1[$ as well as $\ve,\delta>0$ be such that
{\rm(\ref{lazy3})--(\ref{lazy5})} from Lemma~{\rm\ref{lazyla}}
are satisfied. Let $\Omega$ be the set of all bounded, Lipschitz maps $g\colon E\to E$ such that $\max\{\|g\|_\infty,\Lip(g)\}\leq \delta$.
For $g\in\Omega$, let $v_g, w_g\colon E\to E$ be the bounded continuous
maps determined by
\[
(A+g)\circ (\id_E+v_g)\;=\; (\id_E+v_g)\circ A
\]
and $w_g:=(\id_E+v_g)^{-1}-\id_E$. Set $\sigma(g):=v_g$, $\tau(g):=w_g$.
Then
$\sigma$ is Lipschitz as map from $(\Omega,d_\infty)$ to $(BC(E,E),d_\infty)$,
and $\tau\colon (\Omega,d_\infty)\to(BC(E,E),d_\infty)$
is H\"{o}lder of exponent $\alpha$.
\end{thm}
\begin{proof}
Throughout the proof, we equip $BC(E,E)$ and $\Omega$ with the supremum metric $d_\infty$.
Moreover, we give $\Omega\times BC(E,E)$ the metric
$d$ defined via
$d((g_1,v_1),(g_2, v_2)):=\max\{d_\infty(g_1,g_2), d_\infty(v_1,v_2)\}$.
Given $g\in \Omega$, define $f(g,v):=\theta(v)=(\theta_1(v),\theta_2(v))$ for $v\in BC(E,E)$
as in (\ref{vers3}) and (\ref{vers4}) (applied with $h:=0$).
We claim that
\[
f\colon (\Omega \times BC(E,E),d)\to (BC(E,E),d_\infty)
\]
satisfies the hypotheses of the Lipschitz case of Lemma~\ref{hoedepla}.
If this is true, then the map
$\sigma\colon \Omega \to BC(E,E)$
taking $g\in \Omega$ to the fixed point $\sigma(g):=v_g$ of
$f_g:=f(g,.)\colon BC(E,E)\to BC(E,E)$
is Lipschitz.
To establish the claim, note first that condition (c)
of Lemma~\ref{hoedepla} is satisfied by completeness of $BC(E,E)$
(see Remark~\ref{userm}). Condition (b) is satisfied
since (\ref{reusend}) and (\ref{hyp2}) show that
\[
\Lip(f_g)\; \leq \; \max\big\{
\|A^{-1}_2\| (1 +\delta),
\|A_1\| +\delta\big\}\, ,
\]
where the right hand side is $<1$ and independent of $g\in \Omega$.
To see that the maps $f^v:=f(.,v)\colon \Omega\to BC(E,E)$,
for $v\in BC(E,E)$, are uniformly Lipschitz, note that
\[
f^v(g)-f^v(k)=\big(
(g_s-k_s)\circ (\id_E+v)\circ A^{-1},\, A_2^{-1}\circ (k_u-g_u)\circ (\id_E+v)\big)
\]
for $g,k\in BC(E,E)$ and thus
\begin{eqnarray*}
d_\infty(f^v(g),f^v(k)) &\leq& \max\big\{
\|g_s-k_s\|_\infty,\, \|A_2^{-1}\|\, \|k_u-g_u\|\big\}\\
&\leq &
\max\{
1,\,
\|A_2^{-1}\|\}\, d_\infty(k,g)\,.
\end{eqnarray*}
Hence $\Lip(f^v)\leq
\max\{
1,\,
\|A_2^{-1}\|\}$, for all $v\in BC(E,E)$.\\[3mm]
Now define $Y$ as in (\ref{defnY}).
For fixed $h\in \Omega$ and $g:=0$, let $\theta=(\theta_1,\theta_2)$ be as
in (\ref{vers3}) and (\ref{vers4}), and recall from the proof of Lemma~\ref{gethoelder}
that $\theta$ restricts to a contraction
$f_h:=\theta|_Y^Y$ of $Y$.
To see that $\tau$ is H\"older, we need only show that the map $f\colon \Omega\times Y\to Y$,
$f(h,x):= f_h(x)$ satisfies the hypotheses of
Lemma~\ref{hoedepla} (using the metric $d_\infty$ on $Y$ and $d$ on the left hand side).
By the proof of Lemma~\ref{gethoelder},
$Y$ is complete with respect to $d_\infty$.
Thus condition~(c) of Lemma~\ref{hoedepla}
is satisfied, and (b) can be shown as in the first part of this proof.
To verify~(a), let $v\in Y$. For $h,k\in\Omega$,
the first and second components of $f^v(h)-f^v(k)$
are given by
\begin{equation}\label{frstcom}
A_1\circ (v_s\circ(A+h)^{-1}-v_s\circ (A+k)^{-1})
+ k_s\circ  (A+k)^{-1}-h_s\circ (A+h)^{-1}\;\mbox{and}
\end{equation}
\begin{equation}\label{sccom}
A_2^{-1}\circ (h_u-k_u)+A_2^{-1}\circ (v_u\circ (A+h)^{-1}-v_u\circ (A+k)^{-1},
\end{equation}
respectively. The supremum norm of (\ref{frstcom}) is bounded by
\begin{eqnarray}
\lefteqn{\hspace*{-12mm}\|A_1\| \Lip_\alpha(v_s) \|(A+h)^{-1}-(A+k)^{-1}\|_\infty^\alpha
+ \|k_s-h_s\|_\infty\qquad\qquad}\notag \\
& & + \, \Lip_\alpha(h_s) \|(A+k)^{-1}-(A+h)^{-1}\|_\infty^\alpha,\label{intmstep}
\end{eqnarray}
where
$\Lip_\alpha(h_s)\leq \max\{\Lip(h_s),2\|h_s\|_\infty\}\leq 2\delta$ by Lemma~\ref{desc},
$\|k_s-h_s\|_\infty\leq \rho \|k_s-h_s\|_\infty^\alpha$ with $\rho:= \max\{1,2\delta\}$
and
\[
\|(A+k)^{-1}-(A+h)^{-1}\|_\infty\leq\frac{\|A^{-1}\|}{1-\delta\|A^{-1}\|}\|k-h\|_\infty
\]
by Lemma~\ref{globinvpar}.
Hence the following is an upper bound for (\ref{intmstep}):
\begin{equation}\label{upbou}
(\|A_1\| \ve +2\delta)\left(\frac{\|A^{-1}\|}{1-\delta\|A^{-1}\|}\right)^\alpha \|k-h\|_\infty^\alpha
+\rho\|k-h\|_\infty^\alpha\,.
\end{equation}
Likewise, the supremum norm of (\ref{sccom}) is bounded by
\begin{equation}\label{fila}
\|A_2^{-1}\|\rho\|h-k\|_\infty^\alpha
+\|A_2^{-1}\| \underbrace{\Lip_\alpha(v_u)}_{\leq\ve}
\left(\frac{\|A^{-1}\|}{1-\delta\|A^{-1}\|}\right)^\alpha \|k-h\|_\infty^\alpha\,.
\end{equation}
Taking now the maximum of the bounds provided by (\ref{upbou})
and (\ref{fila}), we see that $\|f^v(h)-f^v(k)\|_\infty\leq M\|h-k\|_\infty^\alpha$
for $h,k\in\Omega$,
with some constant $M$ independent of $v$, $h$, and $k$.
\end{proof}
\begin{numba}\label{LNew4}
Let $E$ be a Banach space over $\R$ (equipped with an absolute value $|.|$ equivalent to
the usual one) or an ultrametric field $(\K,|.|)$.
Let $A\colon E\to E$ be a hyperbolic automorphism,
$\|.\|$ be a norm on~$E$ adapted to~$A$
and $\alpha\in \,]0,1[$ as well as $\ve,\delta>0$ be such that
{\rm(\ref{lazy3})--(\ref{lazy5})} from Lemma~{\rm\ref{lazyla}}
are satisfied.
Let $\Omega$, $d_\infty$,
$\sigma\colon g\mto v_g$
and $\tau\colon g\mto w_g$
be as in Theorem~\ref{globparthm}.
If $\K=\R$,
fix a function $\eta$ as in~\ref{introeta}.
Let $r>0$ and
$\wt{\Omega}$ be the set
of all mappings $f\colon B^E_r(0)\to E$
which are strictly differentiable at $0$
with $f(0)=0$ and $f'(0)=A$,
and such that $R_f:=f-A$
is Lipschitz and satisfies the following condition:
\begin{itemize}
\item[(a)]
If $(\K,|.|)$ is ultrametric,
assume that $\Lip(R_f)\leq \delta$.
\item[(b)]
If $\K=\R$, assume that
$\Lip(R_f)\leq \frac{\delta}{1+3\Lip(\eta)}\leq \delta$.
\end{itemize}
The symbol $d_\infty$ will also be used for the supremum metric on~$\wt{\Omega}$.
If $(\K,|.|)$ is ultrametric, let $s:=r$.
If $\K=\R$, let $s=r/3$.
Then (\ref{condpr2}) and (\ref{condpr3}),
respectively,
are satisfied by $R_f$ (in place of $R$),
for all $f\in \wt{\Omega}$.
Define $g_f:=(R_f)_s$
as in (\ref{LNew1}) resp.\ (\ref{LNew2})
(cf.\ also \ref{LNew3}).
Define
\[
\wt{\sigma}(f):=\wt{v}_f:=\sigma(g_f)=v_{g_f}\quad\mbox{and}\quad
\wt{\tau}(f):=\wt{w}_f:=\tau(g_f)=w_{g_f}.
\]
Let $\omega$ be as in (\ref{thetau})
and define $O:=B^E_s(0)$,
$U:=B^E_{\omega^{-1}(s)}(0)$
and $W:=B^E_{\omega^{-1}(\omega^{-1}(s))}(0)$.
\end{numba}
\begin{prop}\label{pardepno2}
In the setting of {\rm\ref{LNew4}},
the map $H_f:=\id_E+\wt{v}_f\colon E\to E$
is a homeomorphism such that $H_f^{-1}=\id_E+\wt{w}_f$,
\begin{equation}\label{forsizes}
W \; \sub \; H_f(U)\; \sub \; O,
\end{equation}
\begin{equation}\label{LNew6}
f\circ H_f|_U=H_f\circ A|_U
\end{equation}
and $H_f(0)=0$.
Moreover,
$\wt{\sigma}\colon f\mto \wt{v}_f$ is Lipschitz as map from $(\wt{\Omega},d_\infty)$ to $(BC(E,E),d_\infty)$,
and $\wt{\tau}\colon (\wt{\Omega},d_\infty)\to(BC(E,E),d_\infty)$, $f\mto \wt{w}_f$
is H\"{o}lder of exponent $\alpha$.
If $\K$ is locally compact and $E$ finite-dimensional,
then also the maps $f\mto \wt{v}_f|_{B^E_t(0)}$
and $f\mto \wt{w}_f|_{B^E_t(0)}$
from $(\wt{\Omega},d_\infty)$ to $(BL_\beta(B^E_t(0),E),\|.\|_\beta)$
are continuous,
for all $\beta<\alpha$ and $t>0$.
\end{prop}
\begin{proof}
That $V:=H_f(U)\sub B^E_s(0)=O$
was verified in \ref{lasttango}.
Since $\Lip_\alpha(\wt{w}_f)\leq \ve$ (cf.\ (\ref{LNew8})) and $\wt{w}_f(0)=0$,
we have $(\id_E+ \wt{w}_f)(B^E_t(0))\sub B^E_{t+\ve t^\alpha}(0)=B^E_{\omega(t)}(0)$
and thus $B^E_t(0)\sub (\id_E+\wt{v}_f)(B^E_{\omega(t)}(0))$.
Choosing $t=\omega^{-1}(\omega^{-1}(s))$, we deduce that
$W \sub H_f(U)$ indeed.\\[2.3mm]
The map
$\Gamma\colon (\wt{\Omega},d_\infty)\to (\Omega,d_\infty)$
is Lipschitz with $\Lip(\Gamma)\leq 1$.
In fact, if $f_1,f_2\in \wt{\Omega}$ and $x\in E$, then
$\|(\Gamma(f_1)-\Gamma(f_2))(x)\|$ equals $0$ or $\|f_1(x)-f_2(x)\|$
or $\xi_s(x)\|f_1(x)-f_2(x)\|$ (with $\xi_s$ as in \ref{introeta}), and hence is bounded by $\|f_1-f_2\|_\infty$
in either case.
Thus
$\wt{\sigma}=\sigma\circ \Gamma$ is Lipschitz
and
$\wt{\tau}=\tau\circ \Gamma$
is H\"{o}lder of exponent~$\alpha$,
by Theorem~\ref{globparthm}
and Lemma~\ref{basecom}.\\[2.4mm]
If $\K$ is locally compact and~$E$ is finite-dimensional,
given $t<0$
let $\cB$ be set of all $u\in BL_\alpha(B^E_t(0),E)$ such that
$\Lip_\alpha(u)\leq \ve$
and $\|u\|_\infty\leq \ve t^\alpha$.
Then $\wt{v}_f|_{B^E_t(0)}\in \cB$ and $\wt{w}_f|_{B^E_t(0)}\in \cB$ for all $f\in \wt{\Omega}$,
as a consequence of~(\ref{LNew8}).
Pick $\beta<\alpha$.
Since $\cB\sub BL_\alpha(B^E_t(0),E)$ is bounded,
$BC(B^E_t(0),E)$ and $BL_\beta(B^E_t(0),E)$ induce
the same topology on~$\cB$, by Lemma~\ref{easgtinhoel}.
Since $f\mto \wt{v}_f$ and $f\mto\wt{w}_f$ are continuous
as maps to $BC(B^E_t(0),E)$ (by the preceding)
and have image in~$\cB$,
we deduce that they are continuous also
as maps to $BL_\beta(B^E_t(0),E)$.
\end{proof}
\begin{numba}\label{nwnumbaba}
Let $(E,\|.\|)$ be a finite-dimensional Banach space over
a locally compact valued field
$(\K,|.|)$, and $U\sub E$ be an open subset.
Recall from \cite[Lemma~3.11]{FIO}
that a function $f\colon U\to E$ is strictly differentiable at each point
if and only if $f$ is $C^1$ in the sense of~\cite{Ber},
i.e., there exists a continuous map $f^{[1]}\colon U^{[1]}\to E$
on the open subset $U^{[1]}:=\{(x,y,t)\in U\times E\times\K\colon x+ty\in U\}$
of $U\times E\times\K$
such that $f^{[1]}(x,y,t)=\frac{1}{t}(f(x+ty)-f(x))$ for all $(x,y,t)\in U^{[1]}$ with $t\not= 0$.
We endow the space $C^1(U,E)$ with the compact-open $C^1$-topology~$\cO_{C_1}$,
i.e., the initial topology with respect to the inclusion map $C^1(U,E)\to C(U,E)_{c.o.}$
and the map $C^1(U,E)\to C(U^{[1]},E)_{c.o.}$, $f\mto f^{[1]}$,
where the spaces on the right-hand side are equipped
with the compact-open topology (see \cite{ZOO}
for further information).
\end{numba}
The proof of the next lemma can be found in Appendix~\ref{nowappD}.
\begin{la}\label{nwnewla}
In {\rm\ref{nwnumbaba}},
let $K\sub U$ be a relatively compact subset.
Then $f|_K\colon K\to E$
is Lipschitz for each $f\in C^1(U,E)$
and $C^1(U,E)
\to [0,\infty[$, $f\mto\Lip(f|_K)$
is a continuous seminorm on $(C^1(U,E),\cO_{C_1})$.
\end{la}
\begin{numba}\label{LNew10}
Assume that the valued field $(\K,|.|)$ is locally compact
and $\K\not\isom \C$.
Let $E$ be a finite-dimensional Banach space over~$\K$
and $P\sub E$ be an open $0$-neighbourhood.
We give
\[
C^1_{**}(P,E):=\{g\in C^1(P,E)\colon \mbox{$g(0)=0$ and $g'(0)=0$}\,\}
\]
the topology $\cO_{C_1}$ induced by $C^1(P,E)$.
Pick $r>0$ such that the compact closure of $K:=B^E_r(0)$ is contained in~$P$.
Then $d_K\colon C^1_{**}(P,E)^2
\to [0,\infty[$, $d_K(g,h):=\|f|_K-g|_K\|_\infty$
is a continuous pseudometric
on $(C^1_{**}(P,E),\cO_{C^1})$.
H\"older and Lipschitz maps between pseudometric spaces
are defined as in the case of metric spaces (recalled in \ref{defbaselip}).
Let $A$, $\|.\|$, $\alpha$, $\ve$, $\delta$, $\wt{\Omega}$,
$\wt{\sigma}$ and $\wt{\tau}$ be as in~\ref{LNew4}.
\end{numba}
\begin{prop}\label{pardepno3}
In the situation of {\rm\ref{LNew10}}, the set
\[
\Omega:=\{g\in C^1_{**}(P,E)\colon A+ g|_{B^E_r(0)}\in \wt{\Omega}\}
\]
is a $0$-neighbourhood in $(C^1_{**}(P,E),\cO_{C_1})$.
The map $\Lambda \colon
(\Omega, d_K) \to(\wt{\Omega},d_\infty)$, $g\mto A+ g|_{B^E_r(0)}$
is Lipschitz with $\Lip(\Lambda)\leq 1$.
The assignment\linebreak
$g\mto \wt{\sigma}(A+g|_{B^E_r(0)})$
defines a Lipschitz map from $(\Omega,d_K)$ to $(BC(E,E),d_\infty)$.
The assigment
$g\mto \wt{\tau}(A+g|_{B^E_r(0)})$
is H\"{o}lder of exponent~$\alpha$ as a mapping
$(\Omega,d_K)\to(BC(E,E),d_\infty)$.
Moreover, the maps
$g\mto \wt{\sigma}(A+g|_{B^E_r(0)})|_{B^E_t(0)}$
and
$g\mto \wt{\tau}(A+g|_{B^E_r(0)})|_{B^E_t(0)}$
are continuous from $(\Omega,d_K)$ to $(BL_\beta(B^E_t(0),E),\|.\|_\beta)$,
for all $t>0$ and $\beta<\alpha$.
\end{prop}
\begin{proof}
If $\K$ is ultrametric, let $\rho:=\delta$.
If $\K=\R$, let $\rho:=\frac{\delta}{1+3\Lip(\eta)}$
(as in \ref{LNew4}).
By Lemma~\ref{nwnewla}, the map $C^1_{**}(P,E)\to[0,\infty[$, $g\mto\Lip(g|_K)$
is a continuous seminorm. Since $\Omega=\{g\in C^1_{**}(P,E)\colon \Lip(g|_K)\leq\rho\}$,
we deduce that $\Omega$ is a $0$-neighbourhood.
We have $\Lambda(\Omega)\sub \wt{\Omega}$ by definition of~$\Omega$,
and the formula $d_\infty(\Lambda(g),\Lambda(h))=\|g|_K-h|_K\|_\infty=d_K(g,h)$
for $g,h\in \Omega$ entails that $\Lip(\Lambda)\leq 1$.
In view of Lemma~\ref{basecom}, the remaining assertions
now follow immediately from Proposition~\ref{pardepno2}.
\end{proof}
\appendix
\section{Existence of ultrametric adapted norms}\label{fiappe}
\begin{la}
Let $(E,\|.\|)$ be an ultrametric Banach space
over a valued field $(\K,|.|)$, and $A\colon E\to E$ be a hyperbolic
automorphism. Then there exists an ultrametric norm
$\|.\|\wt{\;}$ on $E$ adapted to $E$.
\end{la}
\begin{proof}
We first assume that $E=E_\obs$; without loss of generality $E\not=\{0\}$.
Let $\|.\|'$ be a (not necessarily ultrametric) norm on $E$ adapted to $A$.
Since the norms $\|.\|$ and $\|.\|'$ are equivalent, there exists $C\geq 1$ such that
$C^{-1}\|.\|'\leq \|.\|\leq C\|.\|'$. Let $\theta:=\|A\|'<1$ be the operator norm of $A$ with respect to
$\|.\|'$.
Choose an integer $n\geq 2$ so large that $\sigma:=C^2\theta^{n-1}<1$
and define an ultrametric norm $\|.\|\wt{\;}$ on $E$ equivalent to $\|.\|$ via
\[
\|x\|\wt{\;}\; :=\; \max\{\theta^{-\frac{k}{n-1}}\|A^kx\|\colon k=0,\ldots, n-1\}\,.
\]
The operator norm $\|A^n\|$ of $A^n$ with respect to $\|.\|$
satisfies $\|A^n\|\leq C^2 \|A^n\|'\leq C^2(\|A\|')^n=C^2\theta^n$.
To see that $\|.\|\wt{\;}$ is adapted, let $x\in E$.
Then $\|Ax\|\wt{\;}$ is the maximum of
$\max\{\theta^{\frac{1-k}{n-1}}\|A^kx\|\colon k=1,\ldots, n-1\}\leq \theta^{\frac{1}{n-1}}\|x\|\wt{\;}$
and $
\theta^{-\frac{n-1}{n-1}}\|A^nx\|\leq C^2\theta^{n-1}\|x\|\leq C^2\theta^{n-1}\|x\|\wt{\;}$.
By the preceding,
the operator norm $\|A\|\wt{\;}$ of $A$ with respect to $\|.\|\wt{\;}$
satisfies
\[
\|A\|\wt{\;}\; \leq\; \max\{\theta^{\frac{1}{n-1}},\sigma\}\;<\; 1\,.
\]
Hence $\|.\|\wt{\;}$ is an adapted norm on $E=E_\obs$.\\[5mm]
In a general case, $E=E_\obs\oplus E_{\obu}$,
the preceding arguments provide ultrametric norms
$\|.\|_1$ on $E_\obs$ adapted to $A|_{E_\obs}$ and
$\|.\|_2$ on $E_\obu$ adapted to $A^{-1}|_{E_\obu}$.
Then $\|x+y\|\wt{\;}:=\max\{\|x\|_1,\|y\|_2\}$ for $x\in E_\obs$,
$y\in E_\obu$ defines an ultrametric norm on $E$ adapted to $A$.
\end{proof}
\section{Proof of Lemma~\ref{hoedepla}}\label{nowaB}
If also $z_x$ is a fixed point of $f_x$,
then $d_Y(y_x,z_x)=d_Y(f_x(y_x),f_x(z_x))\leq \lambda\, d_Y(y_x,z_x)$,
whence $d_Y(y_x,z_x)=0$ and hence $z_x=y_x$.
For $v\in X$ and $y\in Y$, we have $f_v^n(y)\to y_v$ as $n\to\infty$
since $d_Y(f_v^n(y),y_v)= d_Y(f_v^n(y),f^n_v(y_v))\leq \lambda^nd_Y(y,y_v)$.
In particular, $f_v^n(y_w)\to y_v$
for each $w\in X$. We claim:
\[
d_Y(f_v^n(y_w),y_w)
\;\leq 
\mu \, d_X(v,w)^\alpha
\sum_{k=0}^{n-1}\lambda^k
\quad \mbox{for all $n\in \N$.}
\]
If this is true, letting $n\to\infty$ we deduce that
\[
d_Y(y_v,y_w)\;\leq\;
\mu \, d_X(w,v)^\alpha
\sum_{k=0}^\infty \lambda^k
= \frac{\mu}{1-\lambda} \, d_X(w,v)^\alpha\,,
\]
as required. If $n=1$, we have
$d_Y(f_v(y_w),y_w)=d_Y(f_v(y_w),f_w(y_w))
=d_Y(f^{y_w}(v),f^{y_w}(w))
\leq
\mu \, d_X(v,w)^\alpha$,
verifying the claim in this case.
Assuming that the claim is true for some $n$,
we obtain
\begin{eqnarray*}
d_Y(f_v^{n+1}(y_w),y_w)
&=& d_Y(f_v^{n+1}(y_w),f_w(y_w))\\
&\leq &
d_Y(f_v(f_v^n(y_w)),f_v(y_w))
+
d_Y(f_v(y_w),f_w(y_w))\\
&\leq&
\lambda  \, d_Y(f_v^n(y_w),y_w)
+\mu \, d_X(v,w)^\alpha\\
&\leq&
\lambda
\mu \, d_X(v,w)^\alpha
\sum_{k=0}^{n-1}\lambda^k
+\mu \, d_X(v,w)^\alpha\\
&=&
\mu \, d_X(v,w)^\alpha \sum_{k=0}^n\lambda^k\,,
\end{eqnarray*}
as required. This induction proves the claim.
\section{Proof of Lemma~\ref{compop}}\label{appnowC}
The first assertion is covered by Lemma~\ref{desc}.
Now assume that $\K$ is locally compact (whence $\K$ is $\R$ or $\C$ as a topological
field in the archimedean case),
and assume that~$E$ is finite-dimensional. Then $E\isom \K^n$
(equipped with product topology) for some $n\in \N_0$
as a topological vector space
(see Theorem~2 in \cite[Chapter I, \S2, no.\,3]{BTV}),
whence $E$ is locally compact.\\[2.5mm]
In the real or complex case, define $a:=1$ and $\zeta\colon ]0,\infty[\,\to\,]0,\infty[$, $\zeta(t):=t$.
If $|.|$ and $d$ are ultrametric,
let $a\in \K$
with $0\!<\!|a|\!<\!1$. Define
$\zeta\colon ]0,\infty[\,\to \K$ via
\begin{equation}
\zeta(t):=a^k\quad\mbox{if $\, k\in \Z\,$ and $\,|a|^{k+1}<t\leq |a|^k$.}
\end{equation}
Thus, in either case,
\begin{equation}\label{qualizeta}
|a|\cdot |\zeta(t)|\;<\; t \;\leq\; |\zeta(t)|\quad\mbox{for all $\,t>0$.}
\end{equation}
Let $D:=\{(x,y)\in K\times K\colon x\not=y\}$ and consider the map
\[
D\to \K\,,\quad (x,y)\mto 
\zeta(d(x,y)^\beta)\,.
\]
The continuity of this map is obvious in the real and complex cases.
In the ultrametric case, continuity follows from the fact that
\[
\{(x,y)\in K\times K\colon d(x,y)=t\}
\]
is open in $K\times K$ for each $t>0$
(cf.\ (\ref{stronger})).\\[2.5mm]
We
equip $C(K,E)$ with $\|.\|_\infty$, let
$\phi_1\colon L_\beta(K,E)\to C(K,E)$ be the inclusion map,
and define
\[
\phi_2\colon L_\beta(K,E)\to BC(D,E)
\]
via $\phi_2(f)(x,y):=\frac{f(y)-f(x)}{\zeta(d(y,x)^\beta)}$.
As a consequence of (\ref{qualizeta}),
\[
\|\phi_2(f)\|_\infty\;\leq\; \Lip_\beta(f)\;\leq\; |a|^{-1}\|\phi_2(f)\|_\infty\quad\mbox{for each $\,f\in L_\beta(K,E)$, whence}
\]
\[
\phi=(\phi_1,\phi_2) \colon L_\beta(K,F)\to C(K,E)\times BC(D,E)
\]
is a topological embedding. Moreover $\phi$ has closed image.
To see this, suppose that $\phi(f_n)\to (f,g)$ as $n\to\infty$.
Then
\[
g(x,y)=\lim_{n\to\infty}\phi_2(f_n)(x,y)
=\lim_{n\to\infty}\frac{f_n(y)-f_n(x)}{\zeta(d(y,x)^\beta)}
=\frac{f(y)-f(x)}{\zeta(d(y,x)^\beta)},
\]
entailing that $f$ is H\"{o}lder with $\Lip_\beta(f)\leq |a|^{-1}\|g\|_\infty$,
and $g=\phi_2(f)$. Thus $(f,g)=\phi(f)$.\\[2.5mm]
Now abbreviate $B :=\{f\in L_\alpha(K,E)\colon \|f\|_\alpha<1\}$,
and let $\wb{B}$ be the closure of~$B$ in $C(K,E)$.
Then
\[
\Lip_\beta(f)\leq \max\{\Lip_\alpha(f),2\|f\|_\infty\}\leq 2\quad\mbox{for all $\, f\in B$,}
\]
using (\ref{desc}).
Given $x\in K$ and $\ve>0$,
let $\delta:=\ve^{\frac{1}{\alpha}}$.
For each $y\in B^K_\delta(x)$ and $f\in B$,
we then have $\|f(y)-f(x)\|\leq\Lip_\alpha(f)\, d(x,y)^\alpha\leq \delta^\alpha=\alpha$.
Thus $B$ is equicontinuous. Since, moreover, $\{f(x)\colon x\in B\}\sub \wb{B}^E_1(0)$
is relatively compact for each $x\in K$, Ascoli's Theorem
shows that $\wb{B}\sub C(K,E)$ is compact. We claim that also $\wb{\phi_2(B)}\sub BC(D,E)$
is compact. If this is true, then $C:=\im(\phi)\cap (\wb{B}\times \wb{\phi_2(B)})$
is compact and hence also $\phi^{-1}(C)$ is compact. Since $B\sub \phi^{-1}(C)$,
this proves the lemma.\\[2.5mm]
To verify the claim, let $\ve>0$ be given. We can choose $\sigma>0$
so small that
\begin{equation}\label{chssig}
2\, \sigma^{\alpha-\beta}\, \leq\, \ve\,.
\end{equation}
We let $D_\sigma$ be the set of all $(x,y)\in K\times K$ such that $\frac{\sigma}{9}\leq d(x,y)\leq 2$.
Since $D_\sigma$ is compact, the continuous map
$\gamma \colon D_\sigma \to\K$, $(x,y)\mto\frac{1}{\zeta(d(x,y)^\beta)}$
is uniformly continuous.
Hence, there exists $\delta>0$ such that
\[
|\gamma(x,y)-\gamma(x', y')|\, \leq \, \ve/3
\]
for all $(x,y), (x',y')\in D_\sigma$
such that $d(x,x') < \delta$ and $d(y,y') < \delta$. After shrinking $\delta$ if necessary,
we may assume that also
\begin{equation}\label{moredelt}
\delta\, \leq\, \sigma/9\quad\mbox{and}\quad
\frac{2\, \delta^\alpha}{(\sigma/3)^\beta}\,\leq\, \ve/3 \,.
\end{equation}
Let $(x,y),
(x',y')\in D$ with $d(x,x)|<\delta$
and $d(y,y')<\delta$. We show that
\begin{equation}\label{remans}
\|\phi_2(f)(x',y')-\phi_2(f)(x,y)\|\,\leq\,\ve\, ,
\end{equation}
for all $f\in B$.
If this is true, then
the function $\phi_2(f)$ is uniformly continuous and hence has a unique continuous
extension $\psi(f)\colon \wb{D}\to E$
to the compact closure $\wb{D}\sub K\times K$.
Letting $(x,y)$ and $(x',y')$ as before pass to limits in $\wb{D}$,
we deduce from  (\ref{remans}) that also
\[
\|\psi(f)(x',y')-\psi(f)(x,y)\|\,\leq\,\ve\, ,
\]
for all $f\in B$, $(x,y)\in \wb{D}$ and $(x',y')\in \wb{D}$
such that $d(x,x')< \delta$ and $d(y,y')<\delta$. Hence
$\Omega:=\{\psi(f)\colon f\in B\}$
is an equicontinuous set of functions in $C(\wb{D},E)$.
Given $(x,y)\in D$, we have
$\|\psi (f)(x,y)\|\leq \Lip_\beta(f)\leq 2$ for each $f\in B$
(and, by continuity, this then also holds for all $(x,y)\in \wb{D}$). 
Hence $\{\psi (f)(x,y)\colon f\in B\}\sub \wb{B}_2^E(0)$
and thus the equicontinuous set $\Omega$ is also pointwise relatively compact.
Hence, by Ascoli's Theorem, $\Omega$ is relatively compact in $C(\wb{D},E)$.
Because the restriction map
\[
C(\wb{D},E)\to BC(D,E)\,\quad h\mto h|_D
\]
is continuous linear and takes $\Omega$ to $B$,
we deduce that also $B$ is relatively compact, as claimed.\\[2.5mm]
It only remains to verify (\ref{remans}).
There are two cases. If $d(y,x)<\sigma/3$,
then $d(y',x') \leq \sigma$ (as we assume that $d(x',x),d(y',y) < \delta\leq \sigma/9$)
and hence
\begin{eqnarray*}
\|\phi_2(f)(x',y')-\phi_2(f)(x,y)\|
&\leq & \|\phi_2(f)(x',y')\|+\|\phi_2(f)(x,y)\|\\
&\leq& \frac{\|f(y')-f(x')\|}{d(y',x')^\beta}+\frac{\|f(y)-f(x)\|}{d(y,x)^\beta}\\
&\leq&  \Lip_\alpha(f)  (d(y', x')^{\alpha-\beta}+d(y,x)^{\alpha-\beta})\\
&\leq&  2\sigma^{\alpha-\beta}\leq \ve\, ,
\end{eqnarray*}
by (\ref{chssig}).
If $d(y,x)\geq \sigma/3$, then $d(y',x')\geq d(y,x)-d(y',y)-d(x',x) \geq \frac{\sigma}{9}$ and
\begin{eqnarray*}
\lefteqn{\|\phi_2(f)(x',y')-\phi_2(f)(x,y)\|}\qquad\qquad\\[1mm]
&\leq&
\frac{\|f(y)-f(x)-f(y')+f(x')\|}{|\zeta(d(y,x)^\beta)|}\\
& & +
\underbrace{\left| \frac{1}{\zeta(d(y',x')^\beta)}- \frac{1}{\zeta(d(y,x)^\beta)}\right| }_{\leq\ve/3}
\underbrace{\|f(y')-f(x')\|}_{\leq 2}\\
&\leq & \frac{1}{d(y,x)^\beta} \Lip_\alpha(f)(d(y,y')^\alpha+d(x,x')^\alpha) +2\ve/3\\[.5mm]
&\leq & \frac{2\delta^\alpha}{(\sigma/3)^\beta} +2\ve/3
\;\leq\; \ve\,,
\end{eqnarray*}
using (\ref{moredelt}) for the final inequality.
\section{Proof of Lemma~\ref{nwnewla}}\label{nowappD}
It suffices to show that the set $P:=\{f\in C^1(U,E)\colon \Lip(f|_K)\leq 1\}$
is a $0$-neighbourhood in $(C^1(U,E),\cO_{C^1})$.
After replacing $K$ with its closure, we may assume that $K$ is compact.
Endow~$K$ with the metric $d(x,y):=\|x-y\|$.
Since~$K$ is compact, we have $s:=\spread(K)<\infty$.
Choose $a\in \K$ such that $0<|a|<1$.
Then
\[
L:=\big\{(x,z,t)\in K\times \wb{B}^E_{\frac{s}{|a|}}(0)\times \wb{B}^\K_{|a|}(0)\colon x+tz\in K\big\}
\]
is a compact subset of $U^{[1]}$
and thus
\[
Q:=\{f\in C^1(U,E)\colon f^{[1]}(K)\sub B^E_{|a|}(0)\}
\]
is a $0$-neighbourhood in $(C^1(U,E),\cO_{C^1})$.
To complete the proof, we now show that $Q\sub P$.
Let $f\in Q$. If $x,y\in U$ such that $x\not=y$,
there is a unique integer $k\in \Z$ such that
\[
|a|^{k+1}<\|y-x\|\leq |a|^k\,.
\]
Define $t:=a^k$. Then $|t|<\frac{1}{|a|}\|y-x\|\leq \frac{s}{|a|}$
and $\|t^{-1}(y-x)\|=\frac{1}{|t|}\|x-y\|\leq 1$.
Since, moreover, $x+t(t^{-1}(y-x))=x+(y-x)=y\in K$,
we see that $(x,t^{-1}(y-x),t)\in L$ and hence
\begin{eqnarray*}
\|f(y)-f(x)\| &= & |t| \, \|t^{-1}(f(x+t(t^{-1}(y-x)))-f(x))\|\\
&=& |t| \, \|f^{[1]}(x, t^{-1}(y-x),t)\|\leq |t|\,|a|\leq \|y-x\|.
\end{eqnarray*}
Thus $\Lip(f|_K)\leq 1$ indeed and thus $f\in P$,
showing that $Q\sub P$.\,\Punkt
{\noindent\footnotesize
{\bf Helge Gl\"{o}ckner}, Universit\"at Paderborn, Institut f\"ur Mathematik,
Warburger Str.\ 100,\\
33098 Paderborn, Germany. E-Mail: {\tt glockner\at{}math.upb.de}
\end{document}